\documentclass[
  a4paper,
  UKenglish,
  cleveref,
  autoref]{lipics-v2019-custom}

\usepackage{stmaryrd} 
\usepackage{bm} 

\usepackage{tikz}
\usetikzlibrary{%
  decorations.markings,
  arrows,
  cd
}
\tikzset{
  dot/.style={
    circle,
    draw=black,
    fill=black,
    minimum size=1mm,
    inner sep=0mm
  },
  larrow/.style n args={1}{
    decoration={
      markings,
      mark=at position 0.5 with {\arrow{<};},
    },
    postaction=decorate
  },
  rarrow/.style n args={1}{
    decoration={
      markings,
      mark=at position 0.5 with {\arrow{>};},
    },
    postaction=decorate
  },
}

\newcommand{\icol}[2]{
  \left(\begin{smallmatrix}#1\\#2\end{smallmatrix}\right)%
}
\newcommand{\opval}[3]{
  \langle #1 \mathop{|} #2 \rangle_{#3}
}

\newcommand{\arr}{\rightarrow}
\newcommand{\act}{\odot}
\newcommand{\Lens}{\mathit{Lens}}
\newcommand{\MLens}{\mathit{MLens}}
\newcommand{\Optic}{\mathit{Optic}}
\newcommand{\Tamb}{\mathit{Tamb}}
\newcommand{\Prof}{\mathit{Prof}}
\newcommand{\strength}{\mathit{strength}}
\newcommand\pure[1]{\ensuremath{%
   {}^{\scriptscriptstyle\bm{\lceil}}\mkern-1mu%
   #1%
   {}^{\scriptscriptstyle\bm{\rceil}}}}

\title{String Diagrams for Optics}
\author{Guillaume Boisseau}{University of Oxford, United Kingdom}{}{https://orcid.org/0000-0001-5244-893X}{Work funded by the EPSRC}
\authorrunning{G. Boisseau}
\Copyright{Guillaume Boisseau}



\acknowledgements{%
  I want to thank the reviewers for the SYCO6 workshop for their thorough
  reviews, and Jeremy Gibbons for his ever helpful comments.
}

\nolinenumbers 

\hideLIPIcs  

\EventEditors{Zena M. Ariola}
\EventNoEds{1}
\EventLongTitle{5th International Conference on Formal Structures for Computation and Deduction (FSCD 2020)}
\EventShortTitle{FSCD 2020}
\EventAcronym{FSCD}
\EventYear{2020}
\EventDate{June 29--July 5, 2020}
\EventLocation{Paris, France}
\EventLogo{}
\SeriesVolume{167}
\ArticleNo{22}

\begin{document}
\maketitle

\begin{abstract}

Optics are a data representation for compositional data access, with
lenses as a popular special case. Hedges has presented a diagrammatic
calculus for lenses, but in a way that does not generalize to other
classes of optic. We present a calculus that works for all optics, not
just lenses; this is done by embedding optics into their presheaf
category, which naturally features string diagrams. We apply our
calculus to the common case of lenses, extend it to effectful lenses,
and explore how the laws of optics manifest in this setting.

\end{abstract}

\hypertarget{introduction}{%
\section{Introduction}\label{introduction}}

\emph{Optics} are a versatile categorical structure. Their best-known
special case, \emph{lenses}, have found uses in a variety of contexts,
from machine learning to game theory
\cite{hedgesLensesPhilosophers2018}. Their more general
instantiations have been studied in the context of bidirectional data
transformations \cite{rileyCategoriesOptics2018}. In all cases,
their main feature of interest is their composability and their peculiar
bidirectional information flow.

In the interest of making them easier to represent and manipulate,
authors often spontaneously use diagrams to construct instances of
optics \cite{profunctor_optics,rileyCategoriesOptics2018}. These
diagrams are usually informal, with one notable exception in the work of
Hedges \cite{coherence_for_lenses_and_open_games} on diagrams for
lenses. Hedges' diagrammatic calculus however assumes a lot of structure
on the underlying categories, in a way that doesn't extend to more
general optics.

Here we propose instead a different approach that embeds optics into a
larger space (namely its presheaf category) that naturally has string
diagrams. Not only does this work for the most general optics, but all
the diagrammatic gadgets follow naturally from the embedding, and it
even allows for useful diagrams that would not be expressible in the
category \(\Optic\) alone.

\hypertarget{background}{%
\section{Background}\label{background}}

We fix a monoidal category \((M, \otimes, I, \lambda, \mu, a)\)
throughout the paper.

We assume readers are familiar with \emph{coends}. For an introduction
to the material relevant to the study of optics, see \cite[Chapter
2]{romanProfunctorOpticsTraversals2020}.

\begin{note*}

We will prefer diagrammatic order for composition, using the symbol
\(\fatsemi\).

\end{note*}

\hypertarget{actegories}{%
\subsection{Actegories}\label{actegories}}

\begin{definition}[{\cite{nlab:actegory}}]

An \(M\)-actegory (contraction of ``action'' and ``category'') is a
category \(C\) equipped with a functor \({\act_C}: M \times C \arr C\)
(the ``action'') and two natural structure isomorphisms
\(\lambda_x: I \act_C x \xrightarrow{\sim} x\) and
\(a_{m,n,x}: (m \otimes n) \act_C x \xrightarrow{\sim} m \act_C (n \act_C x)\)
that satisfy compatibility axioms with the monoidal structure of \(M\).

\end{definition}

We will drop the subscripts when the relevant actegory is clear from
context. The naming of the structure morphisms clashes with those of
\(M\) on purpose:

\begin{proposition}

\(M\) has canonically the structure of an \(M\)-actegory, with
\(\act_M = \otimes\), and \(\lambda\) and \(a\) as the actegory
structure morphisms.

\end{proposition}

In what follows, when we use \(M\) as an \(M\)-actegory, we assume this
canonical structure.

\hypertarget{optics}{%
\subsection{Optics}\label{optics}}

\begin{definition}[{\cite[Proposition 3.1.1]{romanProfunctorOpticsTraversals2020}}]

Given two \(M\)-actegories \(C\) and \(D\), we construct the category
\(\Optic_{C,D}\) as follows: objects are pairs \(\icol{x}{u}\) where
\(x: C\) and \(u: D\), and arrows are elements of the set

\[
\Optic_{C,D}(\icol{x}{u}, \icol{y}{v}) := \int^{m: M} C(x, m \act_C y) \times D(m \act_D v, u)
\]

Given \(\alpha: C(x, m \act_C y)\) and \(\beta: D(m \act_D v, u)\), we
will denote the corresponding arrow by \(\opval{\alpha}{\beta}{m}\).
Composition and identities are defined componentwise in the expected
way; see \cite{romanProfunctorOpticsTraversals2020} for more
details.

\end{definition}

\begin{note*}

Expanding the definition of coends in \(Set\), we get that the coend
above denotes the set of pairs \(\opval{\alpha}{\beta}{m}\) with
\(\alpha: C(x, m \act_C y)\) and \(\beta: D(m \act_D v, u)\), quotiented
by the equation
\(\opval{\alpha \fatsemi (f \act_C y)}{\beta}{m} = \opval{\alpha}{(f \act_D v) \fatsemi \beta}{n}\)
for \(f: M(n, m)\).

\end{note*}

Except in special cases, this category is not monoidal. This prevents us
from having string diagrams in the usual way. We will see how to work
around this limitation in the rest of the paper.

\begin{example}

\label{example_lenses} The canonical example of optics are lenses. They
arise when \(C = D = M\) and the monoidal structure of \(C\) is
cartesian. We get:

\[
\Lens_C(\icol{x}{u}, \icol{y}{v}) := \int^{c: C} C(x, c \times y) \times C(c \times v, u)
\]

While this presentation is pleasantly symmetrical, lenses are usually
described as a pair of functions without this unfamiliar coend. We can
in fact calculate that both presentations are equivalent:

\[
\begin{aligned}
\Lens_C(\icol{x}{u}, \icol{y}{v})
&= \int^{c: C} C(x, c \times y) \times C(c \times v, u)
\\ &\cong \int^{c: C} C(x, y) \times C(x, c) \times C(c \times v, u)
\\ &\cong C(x, y) \times \int^{c: C} C(x, c) \times C(c \times v, u)
\\ &\cong C(x, y) \times C(x \times v, u)
\end{aligned}
\]

We recover the usual formulation: a lens from \(\icol{x}{u}\) to
\(\icol{y}{v}\) is a pair of functions \(get: x \arr y\) and
\(put: x \times v \arr u\). The intuition is that \(get\) extracts some
\(y\) from a datum \(x\), and \(put\) allows replacing that \(y\) by a
new \(v\), yielding an updated datum \(u\). It is often the case that
\(x = u\) and \(y = v\), making this intuition clearer, but having
distinct types allows for more flexibility.

A concrete example of a lens that gives access to a field of a record
can be written in Haskell:

\begin{verbatim}
data Lens x u y v = L (x -> y) (x -> v -> u)

data Person = P { name :: String, address :: String }
personName :: Lens Person Person String String
personName = L get put
   where get (P name _) = name
         put (P _ address) name = P name address
\end{verbatim}

The case for distinct types is well illustrated on tuples:

\begin{verbatim}
tupleSnd :: Lens (a, b) (a, c) b c
tupleSnd = L get put
   where get (_, b) = b
         put (a, _) c = (a, c)
\end{verbatim}

\lipicsEnd

\end{example}

\hypertarget{tambara-modules}{%
\subsection{Tambara Modules}\label{tambara-modules}}

\begin{definition}[{\cite[Proposition 5.1.1]{romanProfunctorOpticsTraversals2020}}]

Given two \(M\)-actegories \(C\) and \(D\), we construct the category
\(\Tamb_{C,D}\) as follows: objects are (pro)functors
\(P: C^{op} \times D \arr Set\) equipped with a natural transformation
\(\strength : \int_{m: M} P(a, b) \arr P(m \act_C a, m \act_D b)\)
compatible with the actegory structures; arrows are
\(\strength\)-preserving natural transformations.

\end{definition}

This generalizes the usual notion of strength for a profunctor.

\begin{definition}

We construct the bicategory \(\Tamb\) as follows: objects are
\(M\)-actegories; Hom-categories are the categories \(\Tamb_{C,D}\).

It inherits its bicategorical structure from the bicategory \(\Prof\) of
profunctors: the identities are the hom-profunctors
\(C(\mathord{-}, \mathord{=})\), and the tensor (horizontal composition)
is profunctor composition, defined as usual as follows:

\[
(P \otimes Q)(a, c) = \int^{b} P(a, b) \times Q(b, c)
\]

\end{definition}

\begin{note*}

\(\Prof\) and \(\Tamb\) share in fact a lot of structure. In a sense
\(\Tamb\) is the analogue of \(\Prof\) for \(M\)-actegories, and we will
see that like \(\Prof\) it supports a rich diagrammatic calculus.

\end{note*}

Our interest in Tambara modules comes from the following strong
relationship with optics:

\begin{theorem}[{\cite[Proposition 5.5.2]{romanProfunctorOpticsTraversals2020}}]

\label{psh_op_is_tamb} \([\Optic_{C, D}^{op}, Set] \cong \Tamb_{C, D}\)

\end{theorem}

\begin{proof}

The proof can be found in \cite[Proposition
5.5.2]{romanProfunctorOpticsTraversals2020}, but initially comes from
\cite[Proposition 6.1]{doubles_for_monoidal_categories} in the
special case where \(M = C = D\), along with more results on the
structure of both of those categories.

\end{proof}

\hypertarget{diagrams-for-tambara-modules}{%
\section{Diagrams for Tambara
Modules}\label{diagrams-for-tambara-modules}}

\hypertarget{basics}{%
\subsection{Basics}\label{basics}}

As in any bicategory, cells in \(\Tamb\) can be represented as diagrams,
as follows:

A 0-cell (an \(M\)-actegory) is represented as a planar region delimited
by the other types of cells. For technical reasons we will not represent
them in what follows, but it should be kept in mind that 1-cells can
only be composed if their types match.

A 1-cell \(P: \Tamb_{C, D}\) is represented as a wire, with \(C\) above
and \(D\) below:

\begin{tikzpicture}\matrix[column sep=3mm]{\node at (-0.1,0.0) [anchor=east] {\ensuremath{P}};
\draw[] (-0.1,0.0) to[out=0, in=180] (0.0,0.0);
\path (0.0,0.0) node[coordinate] (n0) {};
&
&
\node at (0.1,0.0) [anchor=west] {\ensuremath{P}};
\draw[] (0.0,0.0) to[out=0, in=180] (0.1,0.0);
\path (0.0,0.0) node[coordinate] (n1) {};
\\};\draw[] (n0) to[out=0, in=180] (n1);
\end{tikzpicture}

Tensoring (1-cell composition) is vertical juxtaposition (for
\(P: \Tamb_{C, D}\) and \(Q: \Tamb_{D, E}\)):

\begin{equation*}\begin{gathered}\begin{tikzpicture}\matrix[column sep=3mm]{\node at (-0.1,0.0) [anchor=east] {\ensuremath{P\otimes{}Q}};
\draw[] (-0.1,0.0) to[out=0, in=180] (0.0,0.0);
\path (0.0,0.0) node[coordinate] (n0) {};
&
&
\node at (0.1,0.0) [anchor=west] {\ensuremath{P\otimes{}Q}};
\draw[] (0.0,0.0) to[out=0, in=180] (0.1,0.0);
\path (0.0,0.0) node[coordinate] (n1) {};
\\};\draw[] (n0) to[out=0, in=180] (n1);
\end{tikzpicture}\end{gathered}\enskip =\enskip\begin{gathered}\begin{tikzpicture}\matrix[column sep=3mm]{\node at (-0.1,0.0) [anchor=east] {\ensuremath{Q}};
\draw[] (-0.1,0.0) to[out=0, in=180] (0.0,0.0);
\path (0.0,0.0) node[coordinate] (n0) {};
\node at (-0.1,0.6) [anchor=east] {\ensuremath{P}};
\draw[] (-0.1,0.6) to[out=0, in=180] (0.0,0.6);
\path (0.0,0.6) node[coordinate] (n1) {};
&
&
\node at (0.1,0.0) [anchor=west] {\ensuremath{Q}};
\draw[] (0.0,0.0) to[out=0, in=180] (0.1,0.0);
\path (0.0,0.0) node[coordinate] (n2) {};
\node at (0.1,0.6) [anchor=west] {\ensuremath{P}};
\draw[] (0.0,0.6) to[out=0, in=180] (0.1,0.6);
\path (0.0,0.6) node[coordinate] (n3) {};
\\};\draw[] (n0) to[out=0, in=180] (n2);
\draw[] (n1) to[out=0, in=180] (n3);
\end{tikzpicture}\end{gathered}\end{equation*}

A 2-cell \(\alpha: P \arr Q\) (for \(P, Q: \Tamb_{C, D}\)) is
represented as:

\begin{tikzpicture}\matrix[column sep=3mm]{\node at (-0.1,0.0) [anchor=east] {\ensuremath{P}};
\draw[] (-0.1,0.0) to[out=0, in=180] (0.0,0.0);
\path (0.0,0.0) node[coordinate] (n0) {};
&
\node at (0.0,0.0) [rectangle, draw, fill=white] (n1) {\ensuremath{\vphantom{\beta} \alpha}};
\path (n1.west) ++(-0.1,0.0) node[coordinate] (n2) {};
\draw[] (n2) to[out=0, in=180] (0.0,0.0);
\path (n1.east) ++(0.1,0.0) node[coordinate] (n3) {};
\draw[] (0.0,0.0) to[out=0, in=180] (n3);
\node at (0.0,0.0) [rectangle, draw, fill=white]  {\ensuremath{\vphantom{\beta} \alpha}};
&
\node at (0.1,0.0) [anchor=west] {\ensuremath{Q}};
\draw[] (0.0,0.0) to[out=0, in=180] (0.1,0.0);
\path (0.0,0.0) node[coordinate] (n4) {};
\\};\draw[] (n0) to[out=0, in=180] (n2);
\draw[] (n3) to[out=0, in=180] (n4);
\end{tikzpicture}

Composition is horizontal juxtaposition:

\begin{equation*}\begin{gathered}\begin{tikzpicture}\matrix[column sep=3mm]{\node at (-0.1,0.0) [anchor=east] {\ensuremath{P}};
\draw[] (-0.1,0.0) to[out=0, in=180] (0.0,0.0);
\path (0.0,0.0) node[coordinate] (n0) {};
&
\node at (0.0,0.0) [rectangle, draw, fill=white] (n1) {\ensuremath{\alpha{}\fatsemi\beta{}}};
\path (n1.west) ++(-0.1,0.0) node[coordinate] (n2) {};
\draw[] (n2) to[out=0, in=180] (0.0,0.0);
\path (n1.east) ++(0.1,0.0) node[coordinate] (n3) {};
\draw[] (0.0,0.0) to[out=0, in=180] (n3);
\node at (0.0,0.0) [rectangle, draw, fill=white]  {\ensuremath{\alpha{}\fatsemi\beta{}}};
&
\node at (0.1,0.0) [anchor=west] {\ensuremath{R}};
\draw[] (0.0,0.0) to[out=0, in=180] (0.1,0.0);
\path (0.0,0.0) node[coordinate] (n4) {};
\\};\draw[] (n0) to[out=0, in=180] (n2);
\draw[] (n3) to[out=0, in=180] (n4);
\end{tikzpicture}\end{gathered}\enskip =\enskip\begin{gathered}\begin{tikzpicture}\matrix[column sep=3mm]{\node at (-0.1,0.0) [anchor=east] {\ensuremath{P}};
\draw[] (-0.1,0.0) to[out=0, in=180] (0.0,0.0);
\path (0.0,0.0) node[coordinate] (n0) {};
&
\node at (0.0,0.0) [rectangle, draw, fill=white] (n1) {\ensuremath{\vphantom{\beta} \alpha}};
\path (n1.west) ++(-0.1,0.0) node[coordinate] (n2) {};
\draw[] (n2) to[out=0, in=180] (0.0,0.0);
\path (n1.east) ++(0.1,0.0) node[coordinate] (n3) {};
\draw[] (0.0,0.0) to[out=0, in=180] (n3);
\node at (0.0,0.0) [rectangle, draw, fill=white]  {\ensuremath{\vphantom{\beta} \alpha}};
&
\node at (0.0,0.0) [rectangle, draw, fill=white] (n4) {\ensuremath{\beta{}}};
\path (n4.west) ++(-0.1,0.0) node[coordinate] (n5) {};
\draw[] (n5) to[out=0, in=180] (0.0,0.0);
\path (n4.east) ++(0.1,0.0) node[coordinate] (n6) {};
\draw[] (0.0,0.0) to[out=0, in=180] (n6);
\node at (0.0,0.0) [rectangle, draw, fill=white]  {\ensuremath{\beta{}}};
&
\node at (0.1,0.0) [anchor=west] {\ensuremath{R}};
\draw[] (0.0,0.0) to[out=0, in=180] (0.1,0.0);
\path (0.0,0.0) node[coordinate] (n7) {};
\\};\draw[] (n0) to[out=0, in=180] (n2);
\draw[] (n3) to[out=0, in=180] (n5);
\draw[] (n6) to[out=0, in=180] (n7);
\end{tikzpicture}\end{gathered}\end{equation*}

and tensoring is vertical juxtaposition:

\begin{equation*}\begin{gathered}\begin{tikzpicture}\matrix[column sep=3mm]{\node at (-0.1,0.0) [anchor=east] {\ensuremath{P\otimes{}R}};
\draw[] (-0.1,0.0) to[out=0, in=180] (0.0,0.0);
\path (0.0,0.0) node[coordinate] (n0) {};
&
\node at (0.0,0.0) [rectangle, draw, fill=white] (n1) {\ensuremath{\alpha{}\otimes{}\beta{}}};
\path (n1.west) ++(-0.1,0.0) node[coordinate] (n2) {};
\draw[] (n2) to[out=0, in=180] (0.0,0.0);
\path (n1.east) ++(0.1,0.0) node[coordinate] (n3) {};
\draw[] (0.0,0.0) to[out=0, in=180] (n3);
\node at (0.0,0.0) [rectangle, draw, fill=white]  {\ensuremath{\alpha{}\otimes{}\beta{}}};
&
\node at (0.1,0.0) [anchor=west] {\ensuremath{Q\otimes{}S}};
\draw[] (0.0,0.0) to[out=0, in=180] (0.1,0.0);
\path (0.0,0.0) node[coordinate] (n4) {};
\\};\draw[] (n0) to[out=0, in=180] (n2);
\draw[] (n3) to[out=0, in=180] (n4);
\end{tikzpicture}\end{gathered}\enskip =\enskip\begin{gathered}\begin{tikzpicture}\matrix[column sep=3mm]{\node at (-0.1,0.0) [anchor=east] {\ensuremath{R}};
\draw[] (-0.1,0.0) to[out=0, in=180] (0.0,0.0);
\path (0.0,0.0) node[coordinate] (n0) {};
\node at (-0.1,0.6) [anchor=east] {\ensuremath{P}};
\draw[] (-0.1,0.6) to[out=0, in=180] (0.0,0.6);
\path (0.0,0.6) node[coordinate] (n1) {};
&
\node at (0.0,0.0) [rectangle, draw, fill=white] (n2) {\ensuremath{\beta{}}};
\path (n2.west) ++(-0.1,0.0) node[coordinate] (n3) {};
\draw[] (n3) to[out=0, in=180] (0.0,0.0);
\path (n2.east) ++(0.1,0.0) node[coordinate] (n4) {};
\draw[] (0.0,0.0) to[out=0, in=180] (n4);
\node at (0.0,0.0) [rectangle, draw, fill=white]  {\ensuremath{\beta{}}};
\node at (0.0,0.6) [rectangle, draw, fill=white] (n5) {\ensuremath{\alpha{}}};
\path (n5.west) ++(-0.1,0.0) node[coordinate] (n6) {};
\draw[] (n6) to[out=0, in=180] (0.0,0.6);
\path (n5.east) ++(0.1,0.0) node[coordinate] (n7) {};
\draw[] (0.0,0.6) to[out=0, in=180] (n7);
\node at (0.0,0.6) [rectangle, draw, fill=white]  {\ensuremath{\alpha{}}};
&
\node at (0.1,0.0) [anchor=west] {\ensuremath{S}};
\draw[] (0.0,0.0) to[out=0, in=180] (0.1,0.0);
\path (0.0,0.0) node[coordinate] (n8) {};
\node at (0.1,0.6) [anchor=west] {\ensuremath{Q}};
\draw[] (0.0,0.6) to[out=0, in=180] (0.1,0.6);
\path (0.0,0.6) node[coordinate] (n9) {};
\\};\draw[] (n0) to[out=0, in=180] (n3);
\draw[] (n1) to[out=0, in=180] (n6);
\draw[] (n4) to[out=0, in=180] (n8);
\draw[] (n7) to[out=0, in=180] (n9);
\end{tikzpicture}\end{gathered}\end{equation*}

For example, one could represent the following complex composition of
cells diagrammatically:

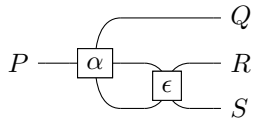
\begin{tikzpicture}\matrix[column sep=3mm]{\node at (-0.1,0.6) [anchor=east] {\ensuremath{P}};
\draw[] (-0.1,0.6) to[out=0, in=180] (0.0,0.6);
\path (0.0,0.6) node[coordinate] (n0) {};
&
\node at (0.0,0.6) [rectangle, draw, fill=white] (n1) {\ensuremath{\alpha{}}};
\path (n1.west) ++(-0.1,0.0) node[coordinate] (n2) {};
\draw[] (n2) to[out=0, in=180] (0.0,0.6);
\path (n1.east) ++(0.1,-0.6) node[coordinate] (n3) {};
\draw[] (0.0,0.6) to[out=down, in=180] (n3);
\path (n1.east) ++(0.1,0.0) node[coordinate] (n4) {};
\draw[] (0.0,0.6) to[out=0, in=180] (n4);
\path (n1.east) ++(0.1,0.6) node[coordinate] (n5) {};
\draw[] (0.0,0.6) to[out=up, in=180] (n5);
\node at (0.0,0.6) [rectangle, draw, fill=white]  {\ensuremath{\alpha{}}};
&
\node at (0.0,0.3) [rectangle, draw, fill=white] (n6) {\ensuremath{\epsilon{}}};
\path (n6.west) ++(-0.1,-0.3) node[coordinate] (n7) {};
\draw[] (n7) to[out=0, in=down] (0.0,0.3);
\path (n6.west) ++(-0.1,0.3) node[coordinate] (n8) {};
\draw[] (n8) to[out=0, in=up] (0.0,0.3);
\path (n6.east) ++(0.1,-0.3) node[coordinate] (n9) {};
\draw[] (0.0,0.3) to[out=down, in=180] (n9);
\path (n6.east) ++(0.1,0.3) node[coordinate] (n10) {};
\draw[] (0.0,0.3) to[out=up, in=180] (n10);
\node at (0.0,0.3) [rectangle, draw, fill=white]  {\ensuremath{\epsilon{}}};
&
\node at (0.1,0.0) [anchor=west] {\ensuremath{S}};
\draw[] (0.0,0.0) to[out=0, in=180] (0.1,0.0);
\path (0.0,0.0) node[coordinate] (n11) {};
\node at (0.1,0.6) [anchor=west] {\ensuremath{R}};
\draw[] (0.0,0.6) to[out=0, in=180] (0.1,0.6);
\path (0.0,0.6) node[coordinate] (n12) {};
\node at (0.1,1.2) [anchor=west] {\ensuremath{Q}};
\draw[] (0.0,1.2) to[out=0, in=180] (0.1,1.2);
\path (0.0,1.2) node[coordinate] (n13) {};
\\};\draw[] (n0) to[out=0, in=180] (n2);
\draw[] (n3) to[out=0, in=180] (n7);
\draw[] (n4) to[out=0, in=180] (n8);
\draw[] (n9) to[out=0, in=180] (n11);
\draw[] (n10) to[out=0, in=180] (n12);
\draw[] (n5) to[out=0, in=180] (n13);
\end{tikzpicture}

The axioms of bicategories ensure that we can interchange boxes like we
do in string diagrams for monoidal categories.

\hypertarget{oriented-wires}{%
\subsection{Oriented Wires}\label{oriented-wires}}

So far, this was common to any bicategory. We can now investigate
gadgets specific to \(\Tamb\).

Let us fix an \(M\)-actegory \(C\).

\begin{definition}

Given \(x: C\), let us define two profunctors
\(R_x := C(\mathord{-}, \mathord{=} \act_C x)\) and
\(L_x := C(\mathord{-} \act_C x, \mathord{=})\).

\end{definition}

\begin{proposition}

\(R_x\) is in \(\Tamb_{C, M}\) and \(L_x\) is in \(\Tamb_{M, C}\), where
\(M\) is taken with its canonical \(M\)-actegory structure.

\end{proposition}

\begin{proof}

\(R_x\) is a profunctor \(C^{op} \times M \arr Set\). The action of the
(\(m \act_C \mathord{-}\)) functor provides it with a strength. The same
works for \(L_x\).

\end{proof}

\begin{proposition}

\(R_x\) extends to a functor \(R: C \arr \Tamb_{C, M}\), and \(L_x\)
extends to a functor \(L: C^{op} \arr \Tamb_{M, C}\)

\end{proposition}

\begin{proof}

Straightforward from their definitions.

\end{proof}

\begin{proposition}

\label{R_respects_actegory} \(R\) and \(L\) respect the actegory
structures: \(R_I \cong L_I \cong M(-,=)\),
\(R_x \otimes R_m \cong R_{m \act x}\), and
\(L_m \otimes L_x \cong L_{m \act x}\).

\end{proposition}

\begin{proof}

See appendix \ref{proof_R_respects_actegory}.

\end{proof}

This justifies the following notation:

\begin{equation}\begin{gathered}\begin{tikzpicture}\matrix[column sep=3mm]{\node at (-0.1,0.0) [anchor=east] {\ensuremath{x}};
\draw[] (-0.1,0.0) to[out=0, in=180] (0.0,0.0);
\path (0.0,0.0) node[coordinate] (n0) {};
&
&
\node at (0.1,0.0) [anchor=west] {\ensuremath{x}};
\draw[] (0.0,0.0) to[out=0, in=180] (0.1,0.0);
\path (0.0,0.0) node[coordinate] (n1) {};
\\};\draw[postaction={decorate}, decoration={markings, mark=at position 0.5 with {\arrow[line width=0.2mm]{angle 90}}}] (n0) to[out=0, in=180] (n1);
\end{tikzpicture}\end{gathered}\enskip :=\enskip\begin{gathered}\begin{tikzpicture}\matrix[column sep=3mm]{\node at (-0.1,0.0) [anchor=east] {\ensuremath{R_x}};
\draw[] (-0.1,0.0) to[out=0, in=180] (0.0,0.0);
\path (0.0,0.0) node[coordinate] (n0) {};
&
&
\node at (0.1,0.0) [anchor=west] {\ensuremath{R_x}};
\draw[] (0.0,0.0) to[out=0, in=180] (0.1,0.0);
\path (0.0,0.0) node[coordinate] (n1) {};
\\};\draw[] (n0) to[out=0, in=180] (n1);
\end{tikzpicture}\end{gathered}\end{equation}

and

\begin{equation}\begin{gathered}\begin{tikzpicture}\matrix[column sep=3mm]{\node at (-0.1,0.0) [anchor=east] {\ensuremath{x}};
\draw[] (-0.1,0.0) to[out=0, in=180] (0.0,0.0);
\path (0.0,0.0) node[coordinate] (n0) {};
&
\node at (0.0,0.0) [rectangle, draw, fill=white] (n1) {\ensuremath{f}};
\path (n1.west) ++(-0.1,0.0) node[coordinate] (n2) {};
\draw[] (n2) to[out=0, in=180] (0.0,0.0);
\path (n1.east) ++(0.1,0.0) node[coordinate] (n3) {};
\draw[] (0.0,0.0) to[out=0, in=180] (n3);
\node at (0.0,0.0) [rectangle, draw, fill=white]  {\ensuremath{f}};
&
\node at (0.1,0.0) [anchor=west] {\ensuremath{y}};
\draw[] (0.0,0.0) to[out=0, in=180] (0.1,0.0);
\path (0.0,0.0) node[coordinate] (n4) {};
\\};\draw[postaction={decorate}, decoration={markings, mark=at position 0.5 with {\arrow[line width=0.2mm]{angle 90}}}] (n0) to[out=0, in=180] (n2);
\draw[postaction={decorate}, decoration={markings, mark=at position 0.5 with {\arrow[line width=0.2mm]{angle 90}}}] (n3) to[out=0, in=180] (n4);
\end{tikzpicture}\end{gathered}\enskip :=\enskip\begin{gathered}\begin{tikzpicture}\matrix[column sep=3mm]{\node at (-0.1,0.0) [anchor=east] {\ensuremath{R_x}};
\draw[] (-0.1,0.0) to[out=0, in=180] (0.0,0.0);
\path (0.0,0.0) node[coordinate] (n0) {};
&
\node at (0.0,0.0) [rectangle, draw, fill=white] (n1) {\ensuremath{R_f}};
\path (n1.west) ++(-0.1,0.0) node[coordinate] (n2) {};
\draw[] (n2) to[out=0, in=180] (0.0,0.0);
\path (n1.east) ++(0.1,0.0) node[coordinate] (n3) {};
\draw[] (0.0,0.0) to[out=0, in=180] (n3);
\node at (0.0,0.0) [rectangle, draw, fill=white]  {\ensuremath{R_f}};
&
\node at (0.1,0.0) [anchor=west] {\ensuremath{R_y}};
\draw[] (0.0,0.0) to[out=0, in=180] (0.1,0.0);
\path (0.0,0.0) node[coordinate] (n4) {};
\\};\draw[] (n0) to[out=0, in=180] (n2);
\draw[] (n3) to[out=0, in=180] (n4);
\end{tikzpicture}\end{gathered}\end{equation}

similarly

\begin{equation}\begin{gathered}\begin{tikzpicture}\matrix[column sep=3mm]{\node at (-0.1,0.0) [anchor=east] {\ensuremath{y}};
\draw[] (-0.1,0.0) to[out=0, in=180] (0.0,0.0);
\path (0.0,0.0) node[coordinate] (n0) {};
&
&
\node at (0.1,0.0) [anchor=west] {\ensuremath{y}};
\draw[] (0.0,0.0) to[out=0, in=180] (0.1,0.0);
\path (0.0,0.0) node[coordinate] (n1) {};
\\};\draw[postaction={decorate}, decoration={markings, mark=at position 0.5 with {\arrow[line width=0.2mm]{angle 90 reversed}}}] (n0) to[out=0, in=180] (n1);
\end{tikzpicture}\end{gathered}\enskip :=\enskip\begin{gathered}\begin{tikzpicture}\matrix[column sep=3mm]{\node at (-0.1,0.0) [anchor=east] {\ensuremath{L_y}};
\draw[] (-0.1,0.0) to[out=0, in=180] (0.0,0.0);
\path (0.0,0.0) node[coordinate] (n0) {};
&
&
\node at (0.1,0.0) [anchor=west] {\ensuremath{L_y}};
\draw[] (0.0,0.0) to[out=0, in=180] (0.1,0.0);
\path (0.0,0.0) node[coordinate] (n1) {};
\\};\draw[] (n0) to[out=0, in=180] (n1);
\end{tikzpicture}\end{gathered}\end{equation}

and

\begin{equation}\begin{gathered}\begin{tikzpicture}\matrix[column sep=3mm]{\node at (-0.1,0.0) [anchor=east] {\ensuremath{y}};
\draw[] (-0.1,0.0) to[out=0, in=180] (0.0,0.0);
\path (0.0,0.0) node[coordinate] (n0) {};
&
\node at (0.0,0.0) [rectangle, draw, fill=white] (n1) {\ensuremath{f}};
\path (n1.west) ++(-0.1,0.0) node[coordinate] (n2) {};
\draw[] (n2) to[out=0, in=180] (0.0,0.0);
\path (n1.east) ++(0.1,0.0) node[coordinate] (n3) {};
\draw[] (0.0,0.0) to[out=0, in=180] (n3);
\node at (0.0,0.0) [rectangle, draw, fill=white]  {\ensuremath{f}};
&
\node at (0.1,0.0) [anchor=west] {\ensuremath{x}};
\draw[] (0.0,0.0) to[out=0, in=180] (0.1,0.0);
\path (0.0,0.0) node[coordinate] (n4) {};
\\};\draw[postaction={decorate}, decoration={markings, mark=at position 0.5 with {\arrow[line width=0.2mm]{angle 90 reversed}}}] (n0) to[out=0, in=180] (n2);
\draw[postaction={decorate}, decoration={markings, mark=at position 0.5 with {\arrow[line width=0.2mm]{angle 90 reversed}}}] (n3) to[out=0, in=180] (n4);
\end{tikzpicture}\end{gathered}\enskip :=\enskip\begin{gathered}\begin{tikzpicture}\matrix[column sep=3mm]{\node at (-0.1,0.0) [anchor=east] {\ensuremath{L_y}};
\draw[] (-0.1,0.0) to[out=0, in=180] (0.0,0.0);
\path (0.0,0.0) node[coordinate] (n0) {};
&
\node at (0.0,0.0) [rectangle, draw, fill=white] (n1) {\ensuremath{L_f}};
\path (n1.west) ++(-0.1,0.0) node[coordinate] (n2) {};
\draw[] (n2) to[out=0, in=180] (0.0,0.0);
\path (n1.east) ++(0.1,0.0) node[coordinate] (n3) {};
\draw[] (0.0,0.0) to[out=0, in=180] (n3);
\node at (0.0,0.0) [rectangle, draw, fill=white]  {\ensuremath{L_f}};
&
\node at (0.1,0.0) [anchor=west] {\ensuremath{L_x}};
\draw[] (0.0,0.0) to[out=0, in=180] (0.1,0.0);
\path (0.0,0.0) node[coordinate] (n4) {};
\\};\draw[] (n0) to[out=0, in=180] (n2);
\draw[] (n3) to[out=0, in=180] (n4);
\end{tikzpicture}\end{gathered}\end{equation}

\begin{note*}

This choice of notation could create confusion as to whether a box on an
oriented wire is meant to be seen as in the image of \(R\)/\(L\) or not.
However we will see later that \(R\) and \(L\) are fully faithful, and
thus this confusion fades away: all boxes on an oriented wire are arrows
in \(C\).

\end{note*}

From the propositions above, we see that this notation respects
composition in \(C\) as well as the \(M\)-actegory structures (note the
inversion that happens when tensoring on a right-oriented wire):

\[\begin{array}{rcl}\begin{gathered}\begin{tikzpicture}\matrix[column sep=3mm]{\node at (-0.1,0.0) [anchor=east] {\ensuremath{m}};
\draw[] (-0.1,0.0) to[out=0, in=180] (0.0,0.0);
\path (0.0,0.0) node[coordinate] (n0) {};
\node at (-0.1,0.6) [anchor=east] {\ensuremath{x}};
\draw[] (-0.1,0.6) to[out=0, in=180] (0.0,0.6);
\path (0.0,0.6) node[coordinate] (n1) {};
&
&
\node at (0.1,0.0) [anchor=west] {\ensuremath{m}};
\draw[] (0.0,0.0) to[out=0, in=180] (0.1,0.0);
\path (0.0,0.0) node[coordinate] (n2) {};
\node at (0.1,0.6) [anchor=west] {\ensuremath{x}};
\draw[] (0.0,0.6) to[out=0, in=180] (0.1,0.6);
\path (0.0,0.6) node[coordinate] (n3) {};
\\};\draw[postaction={decorate}, decoration={markings, mark=at position 0.5 with {\arrow[line width=0.2mm]{angle 90}}}] (n0) to[out=0, in=180] (n2);
\draw[postaction={decorate}, decoration={markings, mark=at position 0.5 with {\arrow[line width=0.2mm]{angle 90}}}] (n1) to[out=0, in=180] (n3);
\end{tikzpicture}\end{gathered}&=&\begin{gathered}\begin{tikzpicture}\matrix[column sep=3mm]{\node at (-0.1,0.0) [anchor=east] {\ensuremath{m \act x}};
\draw[] (-0.1,0.0) to[out=0, in=180] (0.0,0.0);
\path (0.0,0.0) node[coordinate] (n0) {};
&
&
\node at (0.1,0.0) [anchor=west] {\ensuremath{m \act x}};
\draw[] (0.0,0.0) to[out=0, in=180] (0.1,0.0);
\path (0.0,0.0) node[coordinate] (n1) {};
\\};\draw[postaction={decorate}, decoration={markings, mark=at position 0.5 with {\arrow[line width=0.2mm]{angle 90}}}] (n0) to[out=0, in=180] (n1);
\end{tikzpicture}\end{gathered}\\\\\begin{gathered}\begin{tikzpicture}\matrix[column sep=3mm]{\node at (-0.1,0.0) [anchor=east] {\ensuremath{x}};
\draw[] (-0.1,0.0) to[out=0, in=180] (0.0,0.0);
\path (0.0,0.0) node[coordinate] (n0) {};
\node at (-0.1,0.6) [anchor=east] {\ensuremath{m}};
\draw[] (-0.1,0.6) to[out=0, in=180] (0.0,0.6);
\path (0.0,0.6) node[coordinate] (n1) {};
&
&
\node at (0.1,0.0) [anchor=west] {\ensuremath{x}};
\draw[] (0.0,0.0) to[out=0, in=180] (0.1,0.0);
\path (0.0,0.0) node[coordinate] (n2) {};
\node at (0.1,0.6) [anchor=west] {\ensuremath{m}};
\draw[] (0.0,0.6) to[out=0, in=180] (0.1,0.6);
\path (0.0,0.6) node[coordinate] (n3) {};
\\};\draw[postaction={decorate}, decoration={markings, mark=at position 0.5 with {\arrow[line width=0.2mm]{angle 90 reversed}}}] (n0) to[out=0, in=180] (n2);
\draw[postaction={decorate}, decoration={markings, mark=at position 0.5 with {\arrow[line width=0.2mm]{angle 90 reversed}}}] (n1) to[out=0, in=180] (n3);
\end{tikzpicture}\end{gathered}&=&\begin{gathered}\begin{tikzpicture}\matrix[column sep=3mm]{\node at (-0.1,0.0) [anchor=east] {\ensuremath{m \act x}};
\draw[] (-0.1,0.0) to[out=0, in=180] (0.0,0.0);
\path (0.0,0.0) node[coordinate] (n0) {};
&
&
\node at (0.1,0.0) [anchor=west] {\ensuremath{m \act x}};
\draw[] (0.0,0.0) to[out=0, in=180] (0.1,0.0);
\path (0.0,0.0) node[coordinate] (n1) {};
\\};\draw[postaction={decorate}, decoration={markings, mark=at position 0.5 with {\arrow[line width=0.2mm]{angle 90 reversed}}}] (n0) to[out=0, in=180] (n1);
\end{tikzpicture}\end{gathered}\\\\\begin{gathered}\begin{tikzpicture}\matrix[column sep=3mm]{\node at (-0.1,0.0) [anchor=east] {\ensuremath{I}};
\draw[] (-0.1,0.0) to[out=0, in=180] (0.0,0.0);
\path (0.0,0.0) node[coordinate] (n0) {};
&
&
\node at (0.1,0.0) [anchor=west] {\ensuremath{I}};
\draw[] (0.0,0.0) to[out=0, in=180] (0.1,0.0);
\path (0.0,0.0) node[coordinate] (n1) {};
\\};\draw[postaction={decorate}, decoration={markings, mark=at position 0.5 with {\arrow[line width=0.2mm]{angle 90}}}] (n0) to[out=0, in=180] (n1);
\end{tikzpicture}\end{gathered}&=&\begin{gathered}\textit{\small empty diagram}\end{gathered}\\\\\begin{gathered}\begin{tikzpicture}\matrix[column sep=3mm]{\node at (-0.1,0.0) [anchor=east] {\ensuremath{I}};
\draw[] (-0.1,0.0) to[out=0, in=180] (0.0,0.0);
\path (0.0,0.0) node[coordinate] (n0) {};
&
&
\node at (0.1,0.0) [anchor=west] {\ensuremath{I}};
\draw[] (0.0,0.0) to[out=0, in=180] (0.1,0.0);
\path (0.0,0.0) node[coordinate] (n1) {};
\\};\draw[postaction={decorate}, decoration={markings, mark=at position 0.5 with {\arrow[line width=0.2mm]{angle 90 reversed}}}] (n0) to[out=0, in=180] (n1);
\end{tikzpicture}\end{gathered}&=&\begin{gathered}\textit{\small empty diagram}\end{gathered}\end{array}\]

\begin{note*}

Note that because of the types of the 1-cells (that are not shown in the
diagrams), not all tensorings of the oriented wires are allowed. For
example, it could be tempting to think that
\(R_x \otimes R_y \cong R_{y \otimes x}\) for \(x, y: C\), but not only
is \(C\) not monoidal in general, the tensoring doesn't even type-check
since both \(R_x\) and \(R_y\) are objects of \(\Tamb_{C, M}\).

\end{note*}

\begin{note*}

When \(C\) is chosen to be \(M\), both \(R\) and \(L\) provide a
monoidal embedding of \(M\) into \(\Tamb_{M,M}\); we will see later that
it is also fully faithful. This means that the string diagrams in \(M\)
have two full and faithful embeddings into the string diagrams of
\(\Tamb\), using the oriented wires.

\end{note*}

\hypertarget{bending-wires}{%
\subsection{Bending Wires}\label{bending-wires}}

\begin{proposition}

\label{prop_bending_wires} For a given \(x: C\), the modules \(R_x\) and
\(L_x\) are adjoint. Moreover, the structure maps of the adjunction are
dinatural in \(x\).

\end{proposition}

\begin{proof}

\(R_x = C(\mathord{-}, \mathord{=} \act x)\) and
\(L_x = C(\mathord{-} \act x, \mathord{=})\) are clearly adjoint in
\(\Prof\). The adjunction lifts to \(\Tamb\); see appendix
\ref{proof_bending_wires}. Dinaturality in \(x\) is straightforward from
the definition of the unit and counit.

\end{proof}

This means that there exist two 2-cells, that we will draw as:

\begin{equation*}\begin{gathered}\begin{tikzpicture}\matrix[column sep=3mm]{&
\node at (0.0,0.3) [] (n0) {\ensuremath{}};
\path (n0.east) ++(0.1,-0.3) node[coordinate] (n1) {};
\draw[] (0.0,0.3) to[out=down, in=180] (n1);
\path (n0.east) ++(0.1,0.3) node[coordinate] (n2) {};
\draw[] (0.0,0.3) to[out=up, in=180] (n2);
\node at (0.0,0.3) []  {\ensuremath{}};
&
\node at (0.1,0.0) [anchor=west] {\ensuremath{x}};
\draw[] (0.0,0.0) to[out=0, in=180] (0.1,0.0);
\path (0.0,0.0) node[coordinate] (n3) {};
\node at (0.1,0.6) [anchor=west] {\ensuremath{x}};
\draw[] (0.0,0.6) to[out=0, in=180] (0.1,0.6);
\path (0.0,0.6) node[coordinate] (n4) {};
\\};\draw[postaction={decorate}, decoration={markings, mark=at position 0.5 with {\arrow[line width=0.2mm]{angle 90}}}] (n1) to[out=0, in=180] (n3);
\draw[postaction={decorate}, decoration={markings, mark=at position 0.5 with {\arrow[line width=0.2mm]{angle 90 reversed}}}] (n2) to[out=0, in=180] (n4);
\end{tikzpicture}\end{gathered}\enskip \text{and}\enskip\begin{gathered}\begin{tikzpicture}\matrix[column sep=3mm]{\node at (-0.1,0.0) [anchor=east] {\ensuremath{x}};
\draw[] (-0.1,0.0) to[out=0, in=180] (0.0,0.0);
\path (0.0,0.0) node[coordinate] (n0) {};
\node at (-0.1,0.6) [anchor=east] {\ensuremath{x}};
\draw[] (-0.1,0.6) to[out=0, in=180] (0.0,0.6);
\path (0.0,0.6) node[coordinate] (n1) {};
&
\node at (0.0,0.3) [] (n2) {\ensuremath{}};
\path (n2.west) ++(-0.1,-0.3) node[coordinate] (n3) {};
\draw[] (n3) to[out=0, in=down] (0.0,0.3);
\path (n2.west) ++(-0.1,0.3) node[coordinate] (n4) {};
\draw[] (n4) to[out=0, in=up] (0.0,0.3);
\node at (0.0,0.3) []  {\ensuremath{}};
&
\\};\draw[postaction={decorate}, decoration={markings, mark=at position 0.5 with {\arrow[line width=0.2mm]{angle 90 reversed}}}] (n0) to[out=0, in=180] (n3);
\draw[postaction={decorate}, decoration={markings, mark=at position 0.5 with {\arrow[line width=0.2mm]{angle 90}}}] (n1) to[out=0, in=180] (n4);
\end{tikzpicture}\end{gathered}\end{equation*}

that satisfy the so-called ``snake equations'':

\begin{equation}\begin{gathered}\begin{tikzpicture}\matrix[column sep=3mm]{\node at (-0.1,1.2) [anchor=east] {\ensuremath{x}};
\draw[] (-0.1,1.2) to[out=0, in=180] (0.0,1.2);
\path (0.0,1.2) node[coordinate] (n0) {};
&
\node at (0.0,0.3) [] (n1) {\ensuremath{}};
\path (n1.east) ++(0.1,-0.3) node[coordinate] (n2) {};
\draw[] (0.0,0.3) to[out=down, in=180] (n2);
\path (n1.east) ++(0.1,0.3) node[coordinate] (n3) {};
\draw[] (0.0,0.3) to[out=up, in=180] (n3);
\node at (0.0,0.3) []  {\ensuremath{}};
&
\node at (0.0,0.9) [] (n4) {\ensuremath{}};
\path (n4.west) ++(-0.1,-0.3) node[coordinate] (n5) {};
\draw[] (n5) to[out=0, in=down] (0.0,0.9);
\path (n4.west) ++(-0.1,0.3) node[coordinate] (n6) {};
\draw[] (n6) to[out=0, in=up] (0.0,0.9);
\node at (0.0,0.9) []  {\ensuremath{}};
&
\node at (0.1,0.0) [anchor=west] {\ensuremath{x}};
\draw[] (0.0,0.0) to[out=0, in=180] (0.1,0.0);
\path (0.0,0.0) node[coordinate] (n7) {};
\\};\draw[postaction={decorate}, decoration={markings, mark=at position 0.5 with {\arrow[line width=0.2mm]{angle 90 reversed}}}] (n3) to[out=0, in=180] (n5);
\draw[postaction={decorate}, decoration={markings, mark=at position 0.5 with {\arrow[line width=0.2mm]{angle 90}}}] (n0) to[out=0, in=180] (n6);
\draw[postaction={decorate}, decoration={markings, mark=at position 0.5 with {\arrow[line width=0.2mm]{angle 90}}}] (n2) to[out=0, in=180] (n7);
\end{tikzpicture}\end{gathered}\enskip =\enskip\begin{gathered}\begin{tikzpicture}\matrix[column sep=3mm]{\node at (-0.1,0.0) [anchor=east] {\ensuremath{x}};
\draw[] (-0.1,0.0) to[out=0, in=180] (0.0,0.0);
\path (0.0,0.0) node[coordinate] (n0) {};
&
&
\node at (0.1,0.0) [anchor=west] {\ensuremath{x}};
\draw[] (0.0,0.0) to[out=0, in=180] (0.1,0.0);
\path (0.0,0.0) node[coordinate] (n1) {};
\\};\draw[postaction={decorate}, decoration={markings, mark=at position 0.5 with {\arrow[line width=0.2mm]{angle 90}}}] (n0) to[out=0, in=180] (n1);
\end{tikzpicture}\end{gathered}\end{equation}

and

\begin{equation}\begin{gathered}\begin{tikzpicture}\matrix[column sep=3mm]{\node at (-0.1,0.0) [anchor=east] {\ensuremath{x}};
\draw[] (-0.1,0.0) to[out=0, in=180] (0.0,0.0);
\path (0.0,0.0) node[coordinate] (n0) {};
&
\node at (0.0,0.9) [] (n1) {\ensuremath{}};
\path (n1.east) ++(0.1,-0.3) node[coordinate] (n2) {};
\draw[] (0.0,0.9) to[out=down, in=180] (n2);
\path (n1.east) ++(0.1,0.3) node[coordinate] (n3) {};
\draw[] (0.0,0.9) to[out=up, in=180] (n3);
\node at (0.0,0.9) []  {\ensuremath{}};
&
\node at (0.0,0.3) [] (n4) {\ensuremath{}};
\path (n4.west) ++(-0.1,-0.3) node[coordinate] (n5) {};
\draw[] (n5) to[out=0, in=down] (0.0,0.3);
\path (n4.west) ++(-0.1,0.3) node[coordinate] (n6) {};
\draw[] (n6) to[out=0, in=up] (0.0,0.3);
\node at (0.0,0.3) []  {\ensuremath{}};
&
\node at (0.1,1.2) [anchor=west] {\ensuremath{x}};
\draw[] (0.0,1.2) to[out=0, in=180] (0.1,1.2);
\path (0.0,1.2) node[coordinate] (n7) {};
\\};\draw[postaction={decorate}, decoration={markings, mark=at position 0.5 with {\arrow[line width=0.2mm]{angle 90 reversed}}}] (n0) to[out=0, in=180] (n5);
\draw[postaction={decorate}, decoration={markings, mark=at position 0.5 with {\arrow[line width=0.2mm]{angle 90}}}] (n2) to[out=0, in=180] (n6);
\draw[postaction={decorate}, decoration={markings, mark=at position 0.5 with {\arrow[line width=0.2mm]{angle 90 reversed}}}] (n3) to[out=0, in=180] (n7);
\end{tikzpicture}\end{gathered}\enskip =\enskip\begin{gathered}\begin{tikzpicture}\matrix[column sep=3mm]{\node at (-0.1,0.0) [anchor=east] {\ensuremath{x}};
\draw[] (-0.1,0.0) to[out=0, in=180] (0.0,0.0);
\path (0.0,0.0) node[coordinate] (n0) {};
&
&
\node at (0.1,0.0) [anchor=west] {\ensuremath{x}};
\draw[] (0.0,0.0) to[out=0, in=180] (0.1,0.0);
\path (0.0,0.0) node[coordinate] (n1) {};
\\};\draw[postaction={decorate}, decoration={markings, mark=at position 0.5 with {\arrow[line width=0.2mm]{angle 90 reversed}}}] (n0) to[out=0, in=180] (n1);
\end{tikzpicture}\end{gathered}\end{equation}

Those maps are additionally dinatural in \(x\), which means we can also
slide \(C\)-arrows around them:

\begin{equation}\begin{gathered}\begin{tikzpicture}\matrix[column sep=3mm]{&
\node at (0.0,0.3) [] (n0) {\ensuremath{}};
\path (n0.east) ++(0.1,-0.3) node[coordinate] (n1) {};
\draw[] (0.0,0.3) to[out=down, in=180] (n1);
\path (n0.east) ++(0.1,0.3) node[coordinate] (n2) {};
\draw[] (0.0,0.3) to[out=up, in=180] (n2);
\node at (0.0,0.3) []  {\ensuremath{}};
&
\node at (0.0,0.6) [rectangle, draw, fill=white] (n3) {\ensuremath{f}};
\path (n3.west) ++(-0.1,0.0) node[coordinate] (n4) {};
\draw[] (n4) to[out=0, in=180] (0.0,0.6);
\path (n3.east) ++(0.1,0.0) node[coordinate] (n5) {};
\draw[] (0.0,0.6) to[out=0, in=180] (n5);
\node at (0.0,0.6) [rectangle, draw, fill=white]  {\ensuremath{f}};
&
\node at (0.1,0.0) [anchor=west] {\ensuremath{y}};
\draw[] (0.0,0.0) to[out=0, in=180] (0.1,0.0);
\path (0.0,0.0) node[coordinate] (n6) {};
\node at (0.1,0.6) [anchor=west] {\ensuremath{x}};
\draw[] (0.0,0.6) to[out=0, in=180] (0.1,0.6);
\path (0.0,0.6) node[coordinate] (n7) {};
\\};\draw[postaction={decorate}, decoration={markings, mark=at position 0.5 with {\arrow[line width=0.2mm]{angle 90 reversed}}}] (n2) to[out=0, in=180] (n4);
\draw[postaction={decorate}, decoration={markings, mark=at position 0.5 with {\arrow[line width=0.2mm]{angle 90}}}] (n1) to[out=0, in=180] (n6);
\draw[postaction={decorate}, decoration={markings, mark=at position 0.5 with {\arrow[line width=0.2mm]{angle 90 reversed}}}] (n5) to[out=0, in=180] (n7);
\end{tikzpicture}\end{gathered}\enskip =\enskip\begin{gathered}\begin{tikzpicture}\matrix[column sep=3mm]{&
\node at (0.0,0.3) [] (n0) {\ensuremath{}};
\path (n0.east) ++(0.1,-0.3) node[coordinate] (n1) {};
\draw[] (0.0,0.3) to[out=down, in=180] (n1);
\path (n0.east) ++(0.1,0.3) node[coordinate] (n2) {};
\draw[] (0.0,0.3) to[out=up, in=180] (n2);
\node at (0.0,0.3) []  {\ensuremath{}};
&
\node at (0.0,0.0) [rectangle, draw, fill=white] (n3) {\ensuremath{f}};
\path (n3.west) ++(-0.1,0.0) node[coordinate] (n4) {};
\draw[] (n4) to[out=0, in=180] (0.0,0.0);
\path (n3.east) ++(0.1,0.0) node[coordinate] (n5) {};
\draw[] (0.0,0.0) to[out=0, in=180] (n5);
\node at (0.0,0.0) [rectangle, draw, fill=white]  {\ensuremath{f}};
&
\node at (0.1,0.0) [anchor=west] {\ensuremath{y}};
\draw[] (0.0,0.0) to[out=0, in=180] (0.1,0.0);
\path (0.0,0.0) node[coordinate] (n6) {};
\node at (0.1,0.6) [anchor=west] {\ensuremath{x}};
\draw[] (0.0,0.6) to[out=0, in=180] (0.1,0.6);
\path (0.0,0.6) node[coordinate] (n7) {};
\\};\draw[postaction={decorate}, decoration={markings, mark=at position 0.5 with {\arrow[line width=0.2mm]{angle 90}}}] (n1) to[out=0, in=180] (n4);
\draw[postaction={decorate}, decoration={markings, mark=at position 0.5 with {\arrow[line width=0.2mm]{angle 90}}}] (n5) to[out=0, in=180] (n6);
\draw[postaction={decorate}, decoration={markings, mark=at position 0.5 with {\arrow[line width=0.2mm]{angle 90 reversed}}}] (n2) to[out=0, in=180] (n7);
\end{tikzpicture}\end{gathered}\end{equation}

and

\begin{equation}\begin{gathered}\begin{tikzpicture}\matrix[column sep=3mm]{\node at (-0.1,0.0) [anchor=east] {\ensuremath{y}};
\draw[] (-0.1,0.0) to[out=0, in=180] (0.0,0.0);
\path (0.0,0.0) node[coordinate] (n0) {};
\node at (-0.1,0.6) [anchor=east] {\ensuremath{x}};
\draw[] (-0.1,0.6) to[out=0, in=180] (0.0,0.6);
\path (0.0,0.6) node[coordinate] (n1) {};
&
\node at (0.0,0.6) [rectangle, draw, fill=white] (n2) {\ensuremath{f}};
\path (n2.west) ++(-0.1,0.0) node[coordinate] (n3) {};
\draw[] (n3) to[out=0, in=180] (0.0,0.6);
\path (n2.east) ++(0.1,0.0) node[coordinate] (n4) {};
\draw[] (0.0,0.6) to[out=0, in=180] (n4);
\node at (0.0,0.6) [rectangle, draw, fill=white]  {\ensuremath{f}};
&
\node at (0.0,0.3) [] (n5) {\ensuremath{}};
\path (n5.west) ++(-0.1,-0.3) node[coordinate] (n6) {};
\draw[] (n6) to[out=0, in=down] (0.0,0.3);
\path (n5.west) ++(-0.1,0.3) node[coordinate] (n7) {};
\draw[] (n7) to[out=0, in=up] (0.0,0.3);
\node at (0.0,0.3) []  {\ensuremath{}};
&
\\};\draw[postaction={decorate}, decoration={markings, mark=at position 0.5 with {\arrow[line width=0.2mm]{angle 90}}}] (n1) to[out=0, in=180] (n3);
\draw[postaction={decorate}, decoration={markings, mark=at position 0.5 with {\arrow[line width=0.2mm]{angle 90 reversed}}}] (n0) to[out=0, in=180] (n6);
\draw[postaction={decorate}, decoration={markings, mark=at position 0.5 with {\arrow[line width=0.2mm]{angle 90}}}] (n4) to[out=0, in=180] (n7);
\end{tikzpicture}\end{gathered}\enskip =\enskip\begin{gathered}\begin{tikzpicture}\matrix[column sep=3mm]{\node at (-0.1,0.0) [anchor=east] {\ensuremath{y}};
\draw[] (-0.1,0.0) to[out=0, in=180] (0.0,0.0);
\path (0.0,0.0) node[coordinate] (n0) {};
\node at (-0.1,0.6) [anchor=east] {\ensuremath{x}};
\draw[] (-0.1,0.6) to[out=0, in=180] (0.0,0.6);
\path (0.0,0.6) node[coordinate] (n1) {};
&
\node at (0.0,0.0) [rectangle, draw, fill=white] (n2) {\ensuremath{f}};
\path (n2.west) ++(-0.1,0.0) node[coordinate] (n3) {};
\draw[] (n3) to[out=0, in=180] (0.0,0.0);
\path (n2.east) ++(0.1,0.0) node[coordinate] (n4) {};
\draw[] (0.0,0.0) to[out=0, in=180] (n4);
\node at (0.0,0.0) [rectangle, draw, fill=white]  {\ensuremath{f}};
&
\node at (0.0,0.3) [] (n5) {\ensuremath{}};
\path (n5.west) ++(-0.1,-0.3) node[coordinate] (n6) {};
\draw[] (n6) to[out=0, in=down] (0.0,0.3);
\path (n5.west) ++(-0.1,0.3) node[coordinate] (n7) {};
\draw[] (n7) to[out=0, in=up] (0.0,0.3);
\node at (0.0,0.3) []  {\ensuremath{}};
&
\\};\draw[postaction={decorate}, decoration={markings, mark=at position 0.5 with {\arrow[line width=0.2mm]{angle 90 reversed}}}] (n0) to[out=0, in=180] (n3);
\draw[postaction={decorate}, decoration={markings, mark=at position 0.5 with {\arrow[line width=0.2mm]{angle 90 reversed}}}] (n4) to[out=0, in=180] (n6);
\draw[postaction={decorate}, decoration={markings, mark=at position 0.5 with {\arrow[line width=0.2mm]{angle 90}}}] (n1) to[out=0, in=180] (n7);
\end{tikzpicture}\end{gathered}\end{equation}

We have discovered an additional property of the diagrammatic language:
oriented arrows can be bent downwards. Note that bending upwards is not
in general possible.

\begin{note*}

In the case of set-based lenses (i.e.~\(C=D=M=Set\) with the cartesian
product), the second of those maps (the ``cap'') was featured in the
calculus of \cite{coherence_for_lenses_and_open_games}. The first
map (the ``cup'') however cannot be expressed in that calculus.

\end{note*}

\hypertarget{embedding-optics}{%
\section{Embedding Optics}\label{embedding-optics}}

\hypertarget{a-representation-theorem}{%
\subsection{A Representation Theorem}\label{a-representation-theorem}}

We will now use this calculus to express optics. Recall from
\autoref{psh_op_is_tamb} that presheaves on optics are equivalent to
Tambara modules. Consequently, the Yoneda embedding
\(Y: \Optic_{C, D} \rightarrow [\Optic_{C, D}^{op}, Set] \cong \Tamb_{C, D}\)
provides a fully faithful embedding of optics into \(\Tamb\). This is
the crucial property that enables our calculus.

\begin{lemma}

\(Y\icol{x}{u} = R_x \otimes L_u\)

\end{lemma}

\begin{proof}

By definition of \(Y\), \(R\) and \(L\), modulo the equivalence of
\autoref{psh_op_is_tamb}.

\end{proof}

Thus \(Y\icol{x}{u}\) has the following nice diagrammatic notation:

\begin{equation}\begin{gathered}\begin{tikzpicture}\matrix[column sep=3mm]{\node at (-0.1,0.0) [anchor=east] {\ensuremath{Y\icol{x}{u}}};
\draw[] (-0.1,0.0) to[out=0, in=180] (0.0,0.0);
\path (0.0,0.0) node[coordinate] (n0) {};
&
&
\node at (0.1,0.0) [anchor=west] {\ensuremath{Y\icol{x}{u}}};
\draw[] (0.0,0.0) to[out=0, in=180] (0.1,0.0);
\path (0.0,0.0) node[coordinate] (n1) {};
\\};\draw[] (n0) to[out=0, in=180] (n1);
\end{tikzpicture}\end{gathered}\enskip =\enskip\begin{gathered}\begin{tikzpicture}\matrix[column sep=3mm]{\node at (-0.1,0.0) [anchor=east] {\ensuremath{u}};
\draw[] (-0.1,0.0) to[out=0, in=180] (0.0,0.0);
\path (0.0,0.0) node[coordinate] (n0) {};
\node at (-0.1,0.6) [anchor=east] {\ensuremath{x}};
\draw[] (-0.1,0.6) to[out=0, in=180] (0.0,0.6);
\path (0.0,0.6) node[coordinate] (n1) {};
&
&
\node at (0.1,0.0) [anchor=west] {\ensuremath{u}};
\draw[] (0.0,0.0) to[out=0, in=180] (0.1,0.0);
\path (0.0,0.0) node[coordinate] (n2) {};
\node at (0.1,0.6) [anchor=west] {\ensuremath{x}};
\draw[] (0.0,0.6) to[out=0, in=180] (0.1,0.6);
\path (0.0,0.6) node[coordinate] (n3) {};
\\};\draw[postaction={decorate}, decoration={markings, mark=at position 0.5 with {\arrow[line width=0.2mm]{angle 90 reversed}}}] (n0) to[out=0, in=180] (n2);
\draw[postaction={decorate}, decoration={markings, mark=at position 0.5 with {\arrow[line width=0.2mm]{angle 90}}}] (n1) to[out=0, in=180] (n3);
\end{tikzpicture}\end{gathered}\end{equation}

From this we deduce the main theorem of this paper:

\begin{theorem}[Representation theorem]

Optics \(l: \Optic_{C,D}(\icol{x}{u}, \icol{y}{v})\) are in bijection
with arrows in \(\Tamb_{C, D}\) of type:

\begin{tikzpicture}\matrix[column sep=3mm]{\node at (-0.1,0.0) [anchor=east] {\ensuremath{u}};
\draw[] (-0.1,0.0) to[out=0, in=180] (0.0,0.0);
\path (0.0,0.0) node[coordinate] (n0) {};
\node at (-0.1,0.6) [anchor=east] {\ensuremath{x}};
\draw[] (-0.1,0.6) to[out=0, in=180] (0.0,0.6);
\path (0.0,0.6) node[coordinate] (n1) {};
&
\node at (0.0,0.3) [rectangle, draw, fill=white] (n2) {\ensuremath{l}};
\path (n2.west) ++(-0.1,-0.3) node[coordinate] (n3) {};
\draw[] (n3) to[out=0, in=down] (0.0,0.3);
\path (n2.west) ++(-0.1,0.3) node[coordinate] (n4) {};
\draw[] (n4) to[out=0, in=up] (0.0,0.3);
\path (n2.east) ++(0.1,-0.3) node[coordinate] (n5) {};
\draw[] (0.0,0.3) to[out=down, in=180] (n5);
\path (n2.east) ++(0.1,0.3) node[coordinate] (n6) {};
\draw[] (0.0,0.3) to[out=up, in=180] (n6);
\node at (0.0,0.3) [rectangle, draw, fill=white]  {\ensuremath{l}};
&
\node at (0.1,0.0) [anchor=west] {\ensuremath{v}};
\draw[] (0.0,0.0) to[out=0, in=180] (0.1,0.0);
\path (0.0,0.0) node[coordinate] (n7) {};
\node at (0.1,0.6) [anchor=west] {\ensuremath{y}};
\draw[] (0.0,0.6) to[out=0, in=180] (0.1,0.6);
\path (0.0,0.6) node[coordinate] (n8) {};
\\};\draw[postaction={decorate}, decoration={markings, mark=at position 0.5 with {\arrow[line width=0.2mm]{angle 90 reversed}}}] (n0) to[out=0, in=180] (n3);
\draw[postaction={decorate}, decoration={markings, mark=at position 0.5 with {\arrow[line width=0.2mm]{angle 90}}}] (n1) to[out=0, in=180] (n4);
\draw[postaction={decorate}, decoration={markings, mark=at position 0.5 with {\arrow[line width=0.2mm]{angle 90 reversed}}}] (n5) to[out=0, in=180] (n7);
\draw[postaction={decorate}, decoration={markings, mark=at position 0.5 with {\arrow[line width=0.2mm]{angle 90}}}] (n6) to[out=0, in=180] (n8);
\end{tikzpicture}

and moreover this bijection is functorial, i.e.~composition of optics
becomes horizontal composition of diagrams and the identity optic is the
identity diagram.

\end{theorem}

\begin{proof}

By full-faithfulness and functoriality of the Yoneda embedding.

\end{proof}

The consequences of this property need stressing: any diagram of this
type represents an optic, \emph{even if it is made of subcomponents that
are not themselves optics}. A parallel can be drawn with complex
numbers: a complex number with no imaginary part represents a real
number, regardless of whether it was constructed (using complex
operations like rotation) from complex numbers that were not themselves
real numbers. In both cases, we can work in this more general space
(complex numbers/Tambara modules) to reason more flexibly about the
simpler objects (reals/optics).

For example, the following diagram is a valid optic, even though several
of its subcomponents are not optics.

\begin{tikzpicture}\matrix[column sep=3mm]{\node at (-0.1,0.0) [anchor=east] {\ensuremath{u}};
\draw[] (-0.1,0.0) to[out=0, in=180] (0.0,0.0);
\path (0.0,0.0) node[coordinate] (n0) {};
\node at (-0.1,1.8) [anchor=east] {\ensuremath{x}};
\draw[] (-0.1,1.8) to[out=0, in=180] (0.0,1.8);
\path (0.0,1.8) node[coordinate] (n1) {};
&
\node at (0.0,0.9) [] (n2) {\ensuremath{}};
\path (n2.east) ++(0.1,-0.3) node[coordinate] (n3) {};
\draw[] (0.0,0.9) to[out=down, in=180] (n3);
\path (n2.east) ++(0.1,0.3) node[coordinate] (n4) {};
\draw[] (0.0,0.9) to[out=up, in=180] (n4);
\node at (0.0,0.9) []  {\ensuremath{}};
&
\node at (0.0,0.3) [rectangle, draw, fill=white] (n5) {\ensuremath{k}};
\path (n5.west) ++(-0.1,-0.3) node[coordinate] (n6) {};
\draw[] (n6) to[out=0, in=down] (0.0,0.3);
\path (n5.west) ++(-0.1,0.3) node[coordinate] (n7) {};
\draw[] (n7) to[out=0, in=up] (0.0,0.3);
\path (n5.east) ++(0.1,-0.3) node[coordinate] (n8) {};
\draw[] (0.0,0.3) to[out=down, in=180] (n8);
\path (n5.east) ++(0.1,0.3) node[coordinate] (n9) {};
\draw[] (0.0,0.3) to[out=up, in=180] (n9);
\node at (0.0,0.3) [rectangle, draw, fill=white]  {\ensuremath{k}};
\node at (0.0,1.5) [rectangle, draw, fill=white] (n10) {\ensuremath{l}};
\path (n10.west) ++(-0.1,-0.3) node[coordinate] (n11) {};
\draw[] (n11) to[out=0, in=down] (0.0,1.5);
\path (n10.west) ++(-0.1,0.3) node[coordinate] (n12) {};
\draw[] (n12) to[out=0, in=up] (0.0,1.5);
\path (n10.east) ++(0.1,-0.3) node[coordinate] (n13) {};
\draw[] (0.0,1.5) to[out=down, in=180] (n13);
\path (n10.east) ++(0.1,0.3) node[coordinate] (n14) {};
\draw[] (0.0,1.5) to[out=up, in=180] (n14);
\node at (0.0,1.5) [rectangle, draw, fill=white]  {\ensuremath{l}};
&
\node at (0.0,1.5) [] (n15) {\ensuremath{}};
\path (n15.west) ++(-0.1,-0.3) node[coordinate] (n16) {};
\draw[] (n16) to[out=0, in=down] (0.0,1.5);
\path (n15.west) ++(-0.1,0.3) node[coordinate] (n17) {};
\draw[] (n17) to[out=0, in=up] (0.0,1.5);
\node at (0.0,1.5) []  {\ensuremath{}};
&
\node at (0.1,0.0) [anchor=west] {\ensuremath{v}};
\draw[] (0.0,0.0) to[out=0, in=180] (0.1,0.0);
\path (0.0,0.0) node[coordinate] (n18) {};
\node at (0.1,0.6) [anchor=west] {\ensuremath{y}};
\draw[] (0.0,0.6) to[out=0, in=180] (0.1,0.6);
\path (0.0,0.6) node[coordinate] (n19) {};
\\};\draw[postaction={decorate}, decoration={markings, mark=at position 0.5 with {\arrow[line width=0.2mm]{angle 90 reversed}}}] (n0) to[out=0, in=180] (n6);
\draw[postaction={decorate}, decoration={markings, mark=at position 0.5 with {\arrow[line width=0.2mm]{angle 90}}}] (n3) to[out=0, in=180] (n7);
\draw[postaction={decorate}, decoration={markings, mark=at position 0.5 with {\arrow[line width=0.2mm]{angle 90 reversed}}}] (n4) to[out=0, in=180] (n11);
\draw[postaction={decorate}, decoration={markings, mark=at position 0.5 with {\arrow[line width=0.2mm]{angle 90}}}] (n1) to[out=0, in=180] (n12);
\draw[postaction={decorate}, decoration={markings, mark=at position 0.5 with {\arrow[line width=0.2mm]{angle 90 reversed}}}] (n13) to[out=0, in=180] (n16);
\draw[postaction={decorate}, decoration={markings, mark=at position 0.5 with {\arrow[line width=0.2mm]{angle 90}}}] (n14) to[out=0, in=180] (n17);
\draw[postaction={decorate}, decoration={markings, mark=at position 0.5 with {\arrow[line width=0.2mm]{angle 90 reversed}}}] (n8) to[out=0, in=180] (n18);
\draw[postaction={decorate}, decoration={markings, mark=at position 0.5 with {\arrow[line width=0.2mm]{angle 90}}}] (n9) to[out=0, in=180] (n19);
\end{tikzpicture}

\hypertarget{simple-arrows}{%
\subsection{Simple Arrows}\label{simple-arrows}}

The simplest optic we can construct is made out of two simple arrows
(i.e.~arrows in the base \(M\)-actegories). This is sometimes called an
\emph{adapter}. Given \(f: C(x, y)\) and \(g: D(v, u)\), we can see from
its type that \(R_f \otimes L_g\) is an optic:

\begin{tikzpicture}\matrix[column sep=3mm]{\node at (-0.1,0.0) [anchor=east] {\ensuremath{u}};
\draw[] (-0.1,0.0) to[out=0, in=180] (0.0,0.0);
\path (0.0,0.0) node[coordinate] (n0) {};
\node at (-0.1,0.6) [anchor=east] {\ensuremath{x}};
\draw[] (-0.1,0.6) to[out=0, in=180] (0.0,0.6);
\path (0.0,0.6) node[coordinate] (n1) {};
&
\node at (0.0,0.0) [rectangle, draw, fill=white] (n2) {\ensuremath{\vphantom{f}g}};
\path (n2.west) ++(-0.1,0.0) node[coordinate] (n3) {};
\draw[] (n3) to[out=0, in=180] (0.0,0.0);
\path (n2.east) ++(0.1,0.0) node[coordinate] (n4) {};
\draw[] (0.0,0.0) to[out=0, in=180] (n4);
\node at (0.0,0.0) [rectangle, draw, fill=white]  {\ensuremath{\vphantom{f}g}};
\node at (0.0,0.6) [rectangle, draw, fill=white] (n5) {\ensuremath{f}};
\path (n5.west) ++(-0.1,0.0) node[coordinate] (n6) {};
\draw[] (n6) to[out=0, in=180] (0.0,0.6);
\path (n5.east) ++(0.1,0.0) node[coordinate] (n7) {};
\draw[] (0.0,0.6) to[out=0, in=180] (n7);
\node at (0.0,0.6) [rectangle, draw, fill=white]  {\ensuremath{f}};
&
\node at (0.1,0.0) [anchor=west] {\ensuremath{v}};
\draw[] (0.0,0.0) to[out=0, in=180] (0.1,0.0);
\path (0.0,0.0) node[coordinate] (n8) {};
\node at (0.1,0.6) [anchor=west] {\ensuremath{y}};
\draw[] (0.0,0.6) to[out=0, in=180] (0.1,0.6);
\path (0.0,0.6) node[coordinate] (n9) {};
\\};\draw[postaction={decorate}, decoration={markings, mark=at position 0.5 with {\arrow[line width=0.2mm]{angle 90 reversed}}}] (n0) to[out=0, in=180] (n3);
\draw[postaction={decorate}, decoration={markings, mark=at position 0.5 with {\arrow[line width=0.2mm]{angle 90}}}] (n1) to[out=0, in=180] (n6);
\draw[postaction={decorate}, decoration={markings, mark=at position 0.5 with {\arrow[line width=0.2mm]{angle 90 reversed}}}] (n4) to[out=0, in=180] (n8);
\draw[postaction={decorate}, decoration={markings, mark=at position 0.5 with {\arrow[line width=0.2mm]{angle 90}}}] (n7) to[out=0, in=180] (n9);
\end{tikzpicture}

\begin{lemma}

\label{diagram_embedding} The optic corresponding to this diagram is
\(\opval{f \fatsemi \lambda^{-1}_y}{\lambda_v \fatsemi g}{I}\).

\end{lemma}

\begin{proof}

By a straightforward calculation; see appendix
\ref{proof_diagram_embedding}.

\end{proof}

The special case of a single simple arrow is particularly interesting:

\begin{theorem}

\label{arrows_ff} All morphisms of type

\begin{tikzpicture}\matrix[column sep=3mm]{\node at (-0.1,0.0) [anchor=east] {\ensuremath{R_x}};
\draw[] (-0.1,0.0) to[out=0, in=180] (0.0,0.0);
\path (0.0,0.0) node[coordinate] (n0) {};
&
\node at (0.0,0.0) [rectangle, draw, fill=white] (n1) {\ensuremath{l}};
\path (n1.west) ++(-0.1,0.0) node[coordinate] (n2) {};
\draw[] (n2) to[out=0, in=180] (0.0,0.0);
\path (n1.east) ++(0.1,0.0) node[coordinate] (n3) {};
\draw[] (0.0,0.0) to[out=0, in=180] (n3);
\node at (0.0,0.0) [rectangle, draw, fill=white]  {\ensuremath{l}};
&
\node at (0.1,0.0) [anchor=west] {\ensuremath{R_y}};
\draw[] (0.0,0.0) to[out=0, in=180] (0.1,0.0);
\path (0.0,0.0) node[coordinate] (n4) {};
\\};\draw[] (n0) to[out=0, in=180] (n2);
\draw[] (n3) to[out=0, in=180] (n4);
\end{tikzpicture}

are of the form

\begin{tikzpicture}\matrix[column sep=3mm]{\node at (-0.1,0.0) [anchor=east] {\ensuremath{x}};
\draw[] (-0.1,0.0) to[out=0, in=180] (0.0,0.0);
\path (0.0,0.0) node[coordinate] (n0) {};
&
\node at (0.0,0.0) [rectangle, draw, fill=white] (n1) {\ensuremath{f}};
\path (n1.west) ++(-0.1,0.0) node[coordinate] (n2) {};
\draw[] (n2) to[out=0, in=180] (0.0,0.0);
\path (n1.east) ++(0.1,0.0) node[coordinate] (n3) {};
\draw[] (0.0,0.0) to[out=0, in=180] (n3);
\node at (0.0,0.0) [rectangle, draw, fill=white]  {\ensuremath{f}};
&
\node at (0.1,0.0) [anchor=west] {\ensuremath{y}};
\draw[] (0.0,0.0) to[out=0, in=180] (0.1,0.0);
\path (0.0,0.0) node[coordinate] (n4) {};
\\};\draw[postaction={decorate}, decoration={markings, mark=at position 0.5 with {\arrow[line width=0.2mm]{angle 90}}}] (n0) to[out=0, in=180] (n2);
\draw[postaction={decorate}, decoration={markings, mark=at position 0.5 with {\arrow[line width=0.2mm]{angle 90}}}] (n3) to[out=0, in=180] (n4);
\end{tikzpicture}

for some unique \(f: C(x, y)\).

Similarly for \(L\) and wires going to the left.

\end{theorem}

\begin{proof}

Since \(L_I \cong M(-,=)\), we have (using a potentially confusing
notation):

\[\begin{array}{rcccl}\begin{gathered}\begin{tikzpicture}\matrix[column sep=3mm]{\node at (-0.1,0.0) [anchor=east] {\ensuremath{R_x}};
\draw[] (-0.1,0.0) to[out=0, in=180] (0.0,0.0);
\path (0.0,0.0) node[coordinate] (n0) {};
&
\node at (0.0,0.0) [rectangle, draw, fill=white] (n1) {\ensuremath{l}};
\path (n1.west) ++(-0.1,0.0) node[coordinate] (n2) {};
\draw[] (n2) to[out=0, in=180] (0.0,0.0);
\path (n1.east) ++(0.1,0.0) node[coordinate] (n3) {};
\draw[] (0.0,0.0) to[out=0, in=180] (n3);
\node at (0.0,0.0) [rectangle, draw, fill=white]  {\ensuremath{l}};
&
\node at (0.1,0.0) [anchor=west] {\ensuremath{R_y}};
\draw[] (0.0,0.0) to[out=0, in=180] (0.1,0.0);
\path (0.0,0.0) node[coordinate] (n4) {};
\\};\draw[] (n0) to[out=0, in=180] (n2);
\draw[] (n3) to[out=0, in=180] (n4);
\end{tikzpicture}\end{gathered}&=&\begin{gathered}\begin{tikzpicture}\matrix[column sep=3mm]{\node at (-0.1,0.0) [anchor=east] {\ensuremath{I}};
\draw[] (-0.1,0.0) to[out=0, in=180] (0.0,0.0);
\path (0.0,0.0) node[coordinate] (n0) {};
\node at (-0.1,0.6) [anchor=east] {\ensuremath{R_x}};
\draw[] (-0.1,0.6) to[out=0, in=180] (0.0,0.6);
\path (0.0,0.6) node[coordinate] (n1) {};
&
\node at (0.0,0.6) [rectangle, draw, fill=white] (n2) {\ensuremath{l}};
\path (n2.west) ++(-0.1,0.0) node[coordinate] (n3) {};
\draw[] (n3) to[out=0, in=180] (0.0,0.6);
\path (n2.east) ++(0.1,0.0) node[coordinate] (n4) {};
\draw[] (0.0,0.6) to[out=0, in=180] (n4);
\node at (0.0,0.6) [rectangle, draw, fill=white]  {\ensuremath{l}};
&
\node at (0.1,0.0) [anchor=west] {\ensuremath{I}};
\draw[] (0.0,0.0) to[out=0, in=180] (0.1,0.0);
\path (0.0,0.0) node[coordinate] (n5) {};
\node at (0.1,0.6) [anchor=west] {\ensuremath{R_y}};
\draw[] (0.0,0.6) to[out=0, in=180] (0.1,0.6);
\path (0.0,0.6) node[coordinate] (n6) {};
\\};\draw[] (n1) to[out=0, in=180] (n3);
\draw[postaction={decorate}, decoration={markings, mark=at position 0.5 with {\arrow[line width=0.2mm]{angle 90 reversed}}}] (n0) to[out=0, in=180] (n5);
\draw[] (n4) to[out=0, in=180] (n6);
\end{tikzpicture}\end{gathered}&=&\begin{gathered}\begin{tikzpicture}\matrix[column sep=3mm]{\node at (-0.1,0.0) [anchor=east] {\ensuremath{I}};
\draw[] (-0.1,0.0) to[out=0, in=180] (0.0,0.0);
\path (0.0,0.0) node[coordinate] (n0) {};
\node at (-0.1,0.6) [anchor=east] {\ensuremath{x}};
\draw[] (-0.1,0.6) to[out=0, in=180] (0.0,0.6);
\path (0.0,0.6) node[coordinate] (n1) {};
&
\node at (0.0,0.6) [rectangle, draw, fill=white] (n2) {\ensuremath{l}};
\path (n2.west) ++(-0.1,0.0) node[coordinate] (n3) {};
\draw[] (n3) to[out=0, in=180] (0.0,0.6);
\path (n2.east) ++(0.1,0.0) node[coordinate] (n4) {};
\draw[] (0.0,0.6) to[out=0, in=180] (n4);
\node at (0.0,0.6) [rectangle, draw, fill=white]  {\ensuremath{l}};
&
\node at (0.1,0.0) [anchor=west] {\ensuremath{I}};
\draw[] (0.0,0.0) to[out=0, in=180] (0.1,0.0);
\path (0.0,0.0) node[coordinate] (n5) {};
\node at (0.1,0.6) [anchor=west] {\ensuremath{y}};
\draw[] (0.0,0.6) to[out=0, in=180] (0.1,0.6);
\path (0.0,0.6) node[coordinate] (n6) {};
\\};\draw[postaction={decorate}, decoration={markings, mark=at position 0.5 with {\arrow[line width=0.2mm]{angle 90}}}] (n1) to[out=0, in=180] (n3);
\draw[postaction={decorate}, decoration={markings, mark=at position 0.5 with {\arrow[line width=0.2mm]{angle 90 reversed}}}] (n0) to[out=0, in=180] (n5);
\draw[postaction={decorate}, decoration={markings, mark=at position 0.5 with {\arrow[line width=0.2mm]{angle 90}}}] (n4) to[out=0, in=180] (n6);
\end{tikzpicture}\end{gathered}\end{array}\]

Thus by the representation theorem, \(l\) can be seen as an optic in
\(\Optic_{C,M}(\icol{x}{I}, \icol{y}{I})\). We then calculate (see
appendix \ref{proof_arrows_ff}) that
\(\Optic_{C,M}(\icol{x}{I}, \icol{y}{I}) \cong C(x, y)\), with the
reverse direction given by the action of \(R\). The proof for \(L\) is
identical.

\end{proof}

\begin{corollary}

\(R\) and \(L\) are fully faithful.

\end{corollary}

\begin{note*}

As pointed out earlier, in the particular case where we choose \(C=D=M\)
(as in the case of lenses), then \(R\) and \(L\) both provide a
fully-faithful \emph{and monoidal} embedding of the arrows in \(M\) into
diagrams.

\end{note*}

\hypertarget{refining-the-representation-theorem}{%
\subsection{Refining the Representation
Theorem}\label{refining-the-representation-theorem}}

Together, simple arrows and the cap are enough to represent any optic as
a string diagram.

\begin{theorem}

\label{representation_theorem} Given \(\alpha: C(x, m \act y)\) and
\(\beta: D(m \act v, u)\), the optic \(l := \opval{\alpha}{\beta}{m}\)
can be represented as follows:

\begin{equation}\begin{gathered}\begin{tikzpicture}\matrix[column sep=3mm]{\node at (-0.1,0.0) [anchor=east] {\ensuremath{u}};
\draw[] (-0.1,0.0) to[out=0, in=180] (0.0,0.0);
\path (0.0,0.0) node[coordinate] (n0) {};
\node at (-0.1,0.6) [anchor=east] {\ensuremath{x}};
\draw[] (-0.1,0.6) to[out=0, in=180] (0.0,0.6);
\path (0.0,0.6) node[coordinate] (n1) {};
&
\node at (0.0,0.3) [rectangle, draw, fill=white] (n2) {\ensuremath{l}};
\path (n2.west) ++(-0.1,-0.3) node[coordinate] (n3) {};
\draw[] (n3) to[out=0, in=down] (0.0,0.3);
\path (n2.west) ++(-0.1,0.3) node[coordinate] (n4) {};
\draw[] (n4) to[out=0, in=up] (0.0,0.3);
\path (n2.east) ++(0.1,-0.3) node[coordinate] (n5) {};
\draw[] (0.0,0.3) to[out=down, in=180] (n5);
\path (n2.east) ++(0.1,0.3) node[coordinate] (n6) {};
\draw[] (0.0,0.3) to[out=up, in=180] (n6);
\node at (0.0,0.3) [rectangle, draw, fill=white]  {\ensuremath{l}};
&
\node at (0.1,0.0) [anchor=west] {\ensuremath{v}};
\draw[] (0.0,0.0) to[out=0, in=180] (0.1,0.0);
\path (0.0,0.0) node[coordinate] (n7) {};
\node at (0.1,0.6) [anchor=west] {\ensuremath{y}};
\draw[] (0.0,0.6) to[out=0, in=180] (0.1,0.6);
\path (0.0,0.6) node[coordinate] (n8) {};
\\};\draw[postaction={decorate}, decoration={markings, mark=at position 0.5 with {\arrow[line width=0.2mm]{angle 90 reversed}}}] (n0) to[out=0, in=180] (n3);
\draw[postaction={decorate}, decoration={markings, mark=at position 0.5 with {\arrow[line width=0.2mm]{angle 90}}}] (n1) to[out=0, in=180] (n4);
\draw[postaction={decorate}, decoration={markings, mark=at position 0.5 with {\arrow[line width=0.2mm]{angle 90 reversed}}}] (n5) to[out=0, in=180] (n7);
\draw[postaction={decorate}, decoration={markings, mark=at position 0.5 with {\arrow[line width=0.2mm]{angle 90}}}] (n6) to[out=0, in=180] (n8);
\end{tikzpicture}\end{gathered}\enskip =\enskip\begin{gathered}\begin{tikzpicture}\matrix[column sep=3mm]{\node at (-0.1,0.3) [anchor=east] {\ensuremath{u}};
\draw[] (-0.1,0.3) to[out=0, in=180] (0.0,0.3);
\path (0.0,0.3) node[coordinate] (n0) {};
\node at (-0.1,1.5) [anchor=east] {\ensuremath{x}};
\draw[] (-0.1,1.5) to[out=0, in=180] (0.0,1.5);
\path (0.0,1.5) node[coordinate] (n1) {};
&
\node at (0.0,0.3) [rectangle, draw, fill=white] (n2) {\ensuremath{\beta{}}};
\path (n2.west) ++(-0.1,0.0) node[coordinate] (n3) {};
\draw[] (n3) to[out=0, in=180] (0.0,0.3);
\path (n2.east) ++(0.1,-0.3) node[coordinate] (n4) {};
\draw[] (0.0,0.3) to[out=down, in=180] (n4);
\path (n2.east) ++(0.1,0.3) node[coordinate] (n5) {};
\draw[] (0.0,0.3) to[out=up, in=180] (n5);
\node at (0.0,0.3) [rectangle, draw, fill=white]  {\ensuremath{\beta{}}};
\node at (0.0,1.5) [rectangle, draw, fill=white] (n6) {\ensuremath{\vphantom{\beta} \alpha}};
\path (n6.west) ++(-0.1,0.0) node[coordinate] (n7) {};
\draw[] (n7) to[out=0, in=180] (0.0,1.5);
\path (n6.east) ++(0.1,-0.3) node[coordinate] (n8) {};
\draw[] (0.0,1.5) to[out=down, in=180] (n8);
\path (n6.east) ++(0.1,0.3) node[coordinate] (n9) {};
\draw[] (0.0,1.5) to[out=up, in=180] (n9);
\node at (0.0,1.5) [rectangle, draw, fill=white]  {\ensuremath{\vphantom{\beta} \alpha}};
&
\node at (0.0,0.9) [] (n10) {\ensuremath{}};
\path (n10.west) ++(-0.1,-0.3) node[coordinate] (n11) {};
\draw[] (n11) to[out=0, in=down] (0.0,0.9);
\path (n10.west) ++(-0.1,0.3) node[coordinate] (n12) {};
\draw[] (n12) to[out=0, in=up] (0.0,0.9);
\node at (0.0,0.9) []  {\ensuremath{}};
&
\node at (0.1,0.0) [anchor=west] {\ensuremath{v}};
\draw[] (0.0,0.0) to[out=0, in=180] (0.1,0.0);
\path (0.0,0.0) node[coordinate] (n13) {};
\node at (0.1,1.8) [anchor=west] {\ensuremath{y}};
\draw[] (0.0,1.8) to[out=0, in=180] (0.1,1.8);
\path (0.0,1.8) node[coordinate] (n14) {};
\\};\draw[postaction={decorate}, decoration={markings, mark=at position 0.5 with {\arrow[line width=0.2mm]{angle 90 reversed}}}] (n0) to[out=0, in=180] (n3);
\draw[postaction={decorate}, decoration={markings, mark=at position 0.5 with {\arrow[line width=0.2mm]{angle 90}}}] (n1) to[out=0, in=180] (n7);
\draw[postaction={decorate}, decoration={markings, mark=at position 0.5 with {\arrow[line width=0.2mm]{angle 90 reversed}}}] (n5) to[out=0, in=180] (n11);
\draw[postaction={decorate}, decoration={markings, mark=at position 0.5 with {\arrow[line width=0.2mm]{angle 90}}}] (n8) to[out=0, in=180] (n12);
\draw[postaction={decorate}, decoration={markings, mark=at position 0.5 with {\arrow[line width=0.2mm]{angle 90 reversed}}}] (n4) to[out=0, in=180] (n13);
\draw[postaction={decorate}, decoration={markings, mark=at position 0.5 with {\arrow[line width=0.2mm]{angle 90}}}] (n9) to[out=0, in=180] (n14);
\end{tikzpicture}\end{gathered}\end{equation}

\end{theorem}

\begin{proof}

By calculating the composition of the pair of simple arrows with the
cap; see appendix \ref{proof_representation_theorem}.

\end{proof}

\begin{note*}

Recall that the pairs \(\opval{\alpha}{\beta}{m}\) are defined modulo an
equivalence relation. How is this compatible with the diagrammatic
notation? The equivalence says that
\(\opval{\alpha \fatsemi (f \act y)}{\beta}{m} = \opval{\alpha}{(f \act v) \fatsemi \beta}{n}\);
diagrammatically, this becomes:

\begin{equation}\begin{gathered}\begin{tikzpicture}\matrix[column sep=3mm]{\node at (-0.1,0.3) [anchor=east] {\ensuremath{u}};
\draw[] (-0.1,0.3) to[out=0, in=180] (0.0,0.3);
\path (0.0,0.3) node[coordinate] (n0) {};
\node at (-0.1,1.5) [anchor=east] {\ensuremath{x}};
\draw[] (-0.1,1.5) to[out=0, in=180] (0.0,1.5);
\path (0.0,1.5) node[coordinate] (n1) {};
&
\node at (0.0,0.3) [rectangle, draw, fill=white] (n2) {\ensuremath{\beta{}}};
\path (n2.west) ++(-0.1,0.0) node[coordinate] (n3) {};
\draw[] (n3) to[out=0, in=180] (0.0,0.3);
\path (n2.east) ++(0.1,-0.3) node[coordinate] (n4) {};
\draw[] (0.0,0.3) to[out=down, in=180] (n4);
\path (n2.east) ++(0.1,0.3) node[coordinate] (n5) {};
\draw[] (0.0,0.3) to[out=up, in=180] (n5);
\node at (0.0,0.3) [rectangle, draw, fill=white]  {\ensuremath{\beta{}}};
\node at (0.0,1.5) [rectangle, draw, fill=white] (n6) {\ensuremath{\vphantom{\beta} \alpha}};
\path (n6.west) ++(-0.1,0.0) node[coordinate] (n7) {};
\draw[] (n7) to[out=0, in=180] (0.0,1.5);
\path (n6.east) ++(0.1,-0.3) node[coordinate] (n8) {};
\draw[] (0.0,1.5) to[out=down, in=180] (n8);
\path (n6.east) ++(0.1,0.3) node[coordinate] (n9) {};
\draw[] (0.0,1.5) to[out=up, in=180] (n9);
\node at (0.0,1.5) [rectangle, draw, fill=white]  {\ensuremath{\vphantom{\beta} \alpha}};
&
\node at (0.0,1.2) [rectangle, draw, fill=white] (n10) {\ensuremath{f}};
\path (n10.west) ++(-0.1,0.0) node[coordinate] (n11) {};
\draw[] (n11) to[out=0, in=180] (0.0,1.2);
\path (n10.east) ++(0.1,0.0) node[coordinate] (n12) {};
\draw[] (0.0,1.2) to[out=0, in=180] (n12);
\node at (0.0,1.2) [rectangle, draw, fill=white]  {\ensuremath{f}};
&
\node at (0.0,0.9) [] (n13) {\ensuremath{}};
\path (n13.west) ++(-0.1,-0.3) node[coordinate] (n14) {};
\draw[] (n14) to[out=0, in=down] (0.0,0.9);
\path (n13.west) ++(-0.1,0.3) node[coordinate] (n15) {};
\draw[] (n15) to[out=0, in=up] (0.0,0.9);
\node at (0.0,0.9) []  {\ensuremath{}};
&
\node at (0.1,0.0) [anchor=west] {\ensuremath{v}};
\draw[] (0.0,0.0) to[out=0, in=180] (0.1,0.0);
\path (0.0,0.0) node[coordinate] (n16) {};
\node at (0.1,1.8) [anchor=west] {\ensuremath{y}};
\draw[] (0.0,1.8) to[out=0, in=180] (0.1,1.8);
\path (0.0,1.8) node[coordinate] (n17) {};
\\};\draw[postaction={decorate}, decoration={markings, mark=at position 0.5 with {\arrow[line width=0.2mm]{angle 90 reversed}}}] (n0) to[out=0, in=180] (n3);
\draw[postaction={decorate}, decoration={markings, mark=at position 0.5 with {\arrow[line width=0.2mm]{angle 90}}}] (n1) to[out=0, in=180] (n7);
\draw[postaction={decorate}, decoration={markings, mark=at position 0.5 with {\arrow[line width=0.2mm]{angle 90}}}] (n8) to[out=0, in=180] (n11);
\draw[postaction={decorate}, decoration={markings, mark=at position 0.5 with {\arrow[line width=0.2mm]{angle 90 reversed}}}] (n5) to[out=0, in=180] (n14);
\draw[postaction={decorate}, decoration={markings, mark=at position 0.5 with {\arrow[line width=0.2mm]{angle 90}}}] (n12) to[out=0, in=180] (n15);
\draw[postaction={decorate}, decoration={markings, mark=at position 0.5 with {\arrow[line width=0.2mm]{angle 90 reversed}}}] (n4) to[out=0, in=180] (n16);
\draw[postaction={decorate}, decoration={markings, mark=at position 0.5 with {\arrow[line width=0.2mm]{angle 90}}}] (n9) to[out=0, in=180] (n17);
\end{tikzpicture}\end{gathered}\enskip =\enskip\begin{gathered}\begin{tikzpicture}\matrix[column sep=3mm]{\node at (-0.1,0.3) [anchor=east] {\ensuremath{u}};
\draw[] (-0.1,0.3) to[out=0, in=180] (0.0,0.3);
\path (0.0,0.3) node[coordinate] (n0) {};
\node at (-0.1,1.5) [anchor=east] {\ensuremath{x}};
\draw[] (-0.1,1.5) to[out=0, in=180] (0.0,1.5);
\path (0.0,1.5) node[coordinate] (n1) {};
&
\node at (0.0,0.3) [rectangle, draw, fill=white] (n2) {\ensuremath{\beta{}}};
\path (n2.west) ++(-0.1,0.0) node[coordinate] (n3) {};
\draw[] (n3) to[out=0, in=180] (0.0,0.3);
\path (n2.east) ++(0.1,-0.3) node[coordinate] (n4) {};
\draw[] (0.0,0.3) to[out=down, in=180] (n4);
\path (n2.east) ++(0.1,0.3) node[coordinate] (n5) {};
\draw[] (0.0,0.3) to[out=up, in=180] (n5);
\node at (0.0,0.3) [rectangle, draw, fill=white]  {\ensuremath{\beta{}}};
\node at (0.0,1.5) [rectangle, draw, fill=white] (n6) {\ensuremath{\vphantom{\beta} \alpha}};
\path (n6.west) ++(-0.1,0.0) node[coordinate] (n7) {};
\draw[] (n7) to[out=0, in=180] (0.0,1.5);
\path (n6.east) ++(0.1,-0.3) node[coordinate] (n8) {};
\draw[] (0.0,1.5) to[out=down, in=180] (n8);
\path (n6.east) ++(0.1,0.3) node[coordinate] (n9) {};
\draw[] (0.0,1.5) to[out=up, in=180] (n9);
\node at (0.0,1.5) [rectangle, draw, fill=white]  {\ensuremath{\vphantom{\beta} \alpha}};
&
\node at (0.0,0.6) [rectangle, draw, fill=white] (n10) {\ensuremath{f}};
\path (n10.west) ++(-0.1,0.0) node[coordinate] (n11) {};
\draw[] (n11) to[out=0, in=180] (0.0,0.6);
\path (n10.east) ++(0.1,0.0) node[coordinate] (n12) {};
\draw[] (0.0,0.6) to[out=0, in=180] (n12);
\node at (0.0,0.6) [rectangle, draw, fill=white]  {\ensuremath{f}};
&
\node at (0.0,0.9) [] (n13) {\ensuremath{}};
\path (n13.west) ++(-0.1,-0.3) node[coordinate] (n14) {};
\draw[] (n14) to[out=0, in=down] (0.0,0.9);
\path (n13.west) ++(-0.1,0.3) node[coordinate] (n15) {};
\draw[] (n15) to[out=0, in=up] (0.0,0.9);
\node at (0.0,0.9) []  {\ensuremath{}};
&
\node at (0.1,0.0) [anchor=west] {\ensuremath{v}};
\draw[] (0.0,0.0) to[out=0, in=180] (0.1,0.0);
\path (0.0,0.0) node[coordinate] (n16) {};
\node at (0.1,1.8) [anchor=west] {\ensuremath{y}};
\draw[] (0.0,1.8) to[out=0, in=180] (0.1,1.8);
\path (0.0,1.8) node[coordinate] (n17) {};
\\};\draw[postaction={decorate}, decoration={markings, mark=at position 0.5 with {\arrow[line width=0.2mm]{angle 90 reversed}}}] (n0) to[out=0, in=180] (n3);
\draw[postaction={decorate}, decoration={markings, mark=at position 0.5 with {\arrow[line width=0.2mm]{angle 90}}}] (n1) to[out=0, in=180] (n7);
\draw[postaction={decorate}, decoration={markings, mark=at position 0.5 with {\arrow[line width=0.2mm]{angle 90 reversed}}}] (n5) to[out=0, in=180] (n11);
\draw[postaction={decorate}, decoration={markings, mark=at position 0.5 with {\arrow[line width=0.2mm]{angle 90 reversed}}}] (n12) to[out=0, in=180] (n14);
\draw[postaction={decorate}, decoration={markings, mark=at position 0.5 with {\arrow[line width=0.2mm]{angle 90}}}] (n8) to[out=0, in=180] (n15);
\draw[postaction={decorate}, decoration={markings, mark=at position 0.5 with {\arrow[line width=0.2mm]{angle 90 reversed}}}] (n4) to[out=0, in=180] (n16);
\draw[postaction={decorate}, decoration={markings, mark=at position 0.5 with {\arrow[line width=0.2mm]{angle 90}}}] (n9) to[out=0, in=180] (n17);
\end{tikzpicture}\end{gathered}\end{equation}

Which we already know holds, by sliding \(f\) along the bent wire!

\end{note*}

\hypertarget{applications}{%
\section{Applications}\label{applications}}

We present two examples of applications of the calculus that illustrate
its expressivity.

\hypertarget{lawful-optics}{%
\subsection{Lawful Optics}\label{lawful-optics}}

One of the most striking consequences of this calculus (and the question
that led to its discovery) is the neatness with which it can express
optic laws.

As originally constructed by the Haskell community
\cite{kmett/lens}, optics were required to abide by certain
round-trip laws that ensure coherence of their operations. Those laws in
particular coincide with very-well-behavedness
\cite{fosterCombinatorsBidirectionalTree2005} in the case of lenses,
which we investigate in more detail in the next section. Riley
formalized those laws in a general form \cite[Section
3]{rileyCategoriesOptics2018}, but the result is rather hard to
manipulate. The string calculus enables an alternative (and equivalent)
description that is purely diagrammatic:

\begin{definition}

An optic \(l: \icol{x}{x} \arr \icol{y}{y}\) is said to be lawful when

\begin{equation}\begin{gathered}\begin{tikzpicture}\matrix[column sep=3mm]{\node at (-0.1,0.0) [anchor=east] {\ensuremath{x}};
\draw[] (-0.1,0.0) to[out=0, in=180] (0.0,0.0);
\path (0.0,0.0) node[coordinate] (n0) {};
\node at (-0.1,0.6) [anchor=east] {\ensuremath{x}};
\draw[] (-0.1,0.6) to[out=0, in=180] (0.0,0.6);
\path (0.0,0.6) node[coordinate] (n1) {};
&
\node at (0.0,0.3) [rectangle, draw, fill=white] (n2) {\ensuremath{l}};
\path (n2.west) ++(-0.1,-0.3) node[coordinate] (n3) {};
\draw[] (n3) to[out=0, in=down] (0.0,0.3);
\path (n2.west) ++(-0.1,0.3) node[coordinate] (n4) {};
\draw[] (n4) to[out=0, in=up] (0.0,0.3);
\path (n2.east) ++(0.1,-0.3) node[coordinate] (n5) {};
\draw[] (0.0,0.3) to[out=down, in=180] (n5);
\path (n2.east) ++(0.1,0.3) node[coordinate] (n6) {};
\draw[] (0.0,0.3) to[out=up, in=180] (n6);
\node at (0.0,0.3) [rectangle, draw, fill=white]  {\ensuremath{l}};
&
\node at (0.0,0.3) [] (n7) {\ensuremath{}};
\path (n7.west) ++(-0.1,-0.3) node[coordinate] (n8) {};
\draw[] (n8) to[out=0, in=down] (0.0,0.3);
\path (n7.west) ++(-0.1,0.3) node[coordinate] (n9) {};
\draw[] (n9) to[out=0, in=up] (0.0,0.3);
\node at (0.0,0.3) []  {\ensuremath{}};
&
\\};\draw[postaction={decorate}, decoration={markings, mark=at position 0.5 with {\arrow[line width=0.2mm]{angle 90 reversed}}}] (n0) to[out=0, in=180] (n3);
\draw[postaction={decorate}, decoration={markings, mark=at position 0.5 with {\arrow[line width=0.2mm]{angle 90}}}] (n1) to[out=0, in=180] (n4);
\draw[postaction={decorate}, decoration={markings, mark=at position 0.5 with {\arrow[line width=0.2mm]{angle 90 reversed}}}] (n5) to[out=0, in=180] (n8);
\draw[postaction={decorate}, decoration={markings, mark=at position 0.5 with {\arrow[line width=0.2mm]{angle 90}}}] (n6) to[out=0, in=180] (n9);
\end{tikzpicture}\end{gathered}\enskip =\enskip\begin{gathered}\begin{tikzpicture}\matrix[column sep=3mm]{\node at (-0.1,0.0) [anchor=east] {\ensuremath{x}};
\draw[] (-0.1,0.0) to[out=0, in=180] (0.0,0.0);
\path (0.0,0.0) node[coordinate] (n0) {};
\node at (-0.1,0.6) [anchor=east] {\ensuremath{x}};
\draw[] (-0.1,0.6) to[out=0, in=180] (0.0,0.6);
\path (0.0,0.6) node[coordinate] (n1) {};
&
\node at (0.0,0.3) [] (n2) {\ensuremath{}};
\path (n2.west) ++(-0.1,-0.3) node[coordinate] (n3) {};
\draw[] (n3) to[out=0, in=down] (0.0,0.3);
\path (n2.west) ++(-0.1,0.3) node[coordinate] (n4) {};
\draw[] (n4) to[out=0, in=up] (0.0,0.3);
\node at (0.0,0.3) []  {\ensuremath{}};
&
\\};\draw[postaction={decorate}, decoration={markings, mark=at position 0.5 with {\arrow[line width=0.2mm]{angle 90 reversed}}}] (n0) to[out=0, in=180] (n3);
\draw[postaction={decorate}, decoration={markings, mark=at position 0.5 with {\arrow[line width=0.2mm]{angle 90}}}] (n1) to[out=0, in=180] (n4);
\end{tikzpicture}\end{gathered}\end{equation}

and

\begin{equation}\begin{gathered}\begin{tikzpicture}\matrix[column sep=3mm]{\node at (-0.1,0.6) [anchor=east] {\ensuremath{x}};
\draw[] (-0.1,0.6) to[out=0, in=180] (0.0,0.6);
\path (0.0,0.6) node[coordinate] (n0) {};
\node at (-0.1,1.2) [anchor=east] {\ensuremath{x}};
\draw[] (-0.1,1.2) to[out=0, in=180] (0.0,1.2);
\path (0.0,1.2) node[coordinate] (n1) {};
&
\node at (0.0,0.9) [rectangle, draw, fill=white] (n2) {\ensuremath{l}};
\path (n2.west) ++(-0.1,-0.3) node[coordinate] (n3) {};
\draw[] (n3) to[out=0, in=down] (0.0,0.9);
\path (n2.west) ++(-0.1,0.3) node[coordinate] (n4) {};
\draw[] (n4) to[out=0, in=up] (0.0,0.9);
\path (n2.east) ++(0.1,-0.9) node[coordinate] (n5) {};
\draw[] (0.0,0.9) to[out=down, in=180] (n5);
\path (n2.east) ++(0.1,0.9) node[coordinate] (n6) {};
\draw[] (0.0,0.9) to[out=up, in=180] (n6);
\node at (0.0,0.9) [rectangle, draw, fill=white]  {\ensuremath{l}};
&
\node at (0.0,0.9) [] (n7) {\ensuremath{}};
\path (n7.east) ++(0.1,-0.3) node[coordinate] (n8) {};
\draw[] (0.0,0.9) to[out=down, in=180] (n8);
\path (n7.east) ++(0.1,0.3) node[coordinate] (n9) {};
\draw[] (0.0,0.9) to[out=up, in=180] (n9);
\node at (0.0,0.9) []  {\ensuremath{}};
&
\node at (0.1,0.0) [anchor=west] {\ensuremath{y}};
\draw[] (0.0,0.0) to[out=0, in=180] (0.1,0.0);
\path (0.0,0.0) node[coordinate] (n10) {};
\node at (0.1,0.6) [anchor=west] {\ensuremath{y}};
\draw[] (0.0,0.6) to[out=0, in=180] (0.1,0.6);
\path (0.0,0.6) node[coordinate] (n11) {};
\node at (0.1,1.2) [anchor=west] {\ensuremath{y}};
\draw[] (0.0,1.2) to[out=0, in=180] (0.1,1.2);
\path (0.0,1.2) node[coordinate] (n12) {};
\node at (0.1,1.8) [anchor=west] {\ensuremath{y}};
\draw[] (0.0,1.8) to[out=0, in=180] (0.1,1.8);
\path (0.0,1.8) node[coordinate] (n13) {};
\\};\draw[postaction={decorate}, decoration={markings, mark=at position 0.5 with {\arrow[line width=0.2mm]{angle 90 reversed}}}] (n0) to[out=0, in=180] (n3);
\draw[postaction={decorate}, decoration={markings, mark=at position 0.5 with {\arrow[line width=0.2mm]{angle 90}}}] (n1) to[out=0, in=180] (n4);
\draw[postaction={decorate}, decoration={markings, mark=at position 0.5 with {\arrow[line width=0.2mm]{angle 90 reversed}}}] (n5) to[out=0, in=180] (n10);
\draw[postaction={decorate}, decoration={markings, mark=at position 0.5 with {\arrow[line width=0.2mm]{angle 90}}}] (n8) to[out=0, in=180] (n11);
\draw[postaction={decorate}, decoration={markings, mark=at position 0.5 with {\arrow[line width=0.2mm]{angle 90 reversed}}}] (n9) to[out=0, in=180] (n12);
\draw[postaction={decorate}, decoration={markings, mark=at position 0.5 with {\arrow[line width=0.2mm]{angle 90}}}] (n6) to[out=0, in=180] (n13);
\end{tikzpicture}\end{gathered}\enskip =\enskip\begin{gathered}\begin{tikzpicture}\matrix[column sep=3mm]{\node at (-0.1,0.0) [anchor=east] {\ensuremath{x}};
\draw[] (-0.1,0.0) to[out=0, in=180] (0.0,0.0);
\path (0.0,0.0) node[coordinate] (n0) {};
\node at (-0.1,1.8) [anchor=east] {\ensuremath{x}};
\draw[] (-0.1,1.8) to[out=0, in=180] (0.0,1.8);
\path (0.0,1.8) node[coordinate] (n1) {};
&
\node at (0.0,0.9) [] (n2) {\ensuremath{}};
\path (n2.east) ++(0.1,-0.3) node[coordinate] (n3) {};
\draw[] (0.0,0.9) to[out=down, in=180] (n3);
\path (n2.east) ++(0.1,0.3) node[coordinate] (n4) {};
\draw[] (0.0,0.9) to[out=up, in=180] (n4);
\node at (0.0,0.9) []  {\ensuremath{}};
&
\node at (0.0,0.3) [rectangle, draw, fill=white] (n5) {\ensuremath{l}};
\path (n5.west) ++(-0.1,-0.3) node[coordinate] (n6) {};
\draw[] (n6) to[out=0, in=down] (0.0,0.3);
\path (n5.west) ++(-0.1,0.3) node[coordinate] (n7) {};
\draw[] (n7) to[out=0, in=up] (0.0,0.3);
\path (n5.east) ++(0.1,-0.3) node[coordinate] (n8) {};
\draw[] (0.0,0.3) to[out=down, in=180] (n8);
\path (n5.east) ++(0.1,0.3) node[coordinate] (n9) {};
\draw[] (0.0,0.3) to[out=up, in=180] (n9);
\node at (0.0,0.3) [rectangle, draw, fill=white]  {\ensuremath{l}};
\node at (0.0,1.5) [rectangle, draw, fill=white] (n10) {\ensuremath{l}};
\path (n10.west) ++(-0.1,-0.3) node[coordinate] (n11) {};
\draw[] (n11) to[out=0, in=down] (0.0,1.5);
\path (n10.west) ++(-0.1,0.3) node[coordinate] (n12) {};
\draw[] (n12) to[out=0, in=up] (0.0,1.5);
\path (n10.east) ++(0.1,-0.3) node[coordinate] (n13) {};
\draw[] (0.0,1.5) to[out=down, in=180] (n13);
\path (n10.east) ++(0.1,0.3) node[coordinate] (n14) {};
\draw[] (0.0,1.5) to[out=up, in=180] (n14);
\node at (0.0,1.5) [rectangle, draw, fill=white]  {\ensuremath{l}};
&
\node at (0.1,0.0) [anchor=west] {\ensuremath{y}};
\draw[] (0.0,0.0) to[out=0, in=180] (0.1,0.0);
\path (0.0,0.0) node[coordinate] (n15) {};
\node at (0.1,0.6) [anchor=west] {\ensuremath{y}};
\draw[] (0.0,0.6) to[out=0, in=180] (0.1,0.6);
\path (0.0,0.6) node[coordinate] (n16) {};
\node at (0.1,1.2) [anchor=west] {\ensuremath{y}};
\draw[] (0.0,1.2) to[out=0, in=180] (0.1,1.2);
\path (0.0,1.2) node[coordinate] (n17) {};
\node at (0.1,1.8) [anchor=west] {\ensuremath{y}};
\draw[] (0.0,1.8) to[out=0, in=180] (0.1,1.8);
\path (0.0,1.8) node[coordinate] (n18) {};
\\};\draw[postaction={decorate}, decoration={markings, mark=at position 0.5 with {\arrow[line width=0.2mm]{angle 90 reversed}}}] (n0) to[out=0, in=180] (n6);
\draw[postaction={decorate}, decoration={markings, mark=at position 0.5 with {\arrow[line width=0.2mm]{angle 90}}}] (n3) to[out=0, in=180] (n7);
\draw[postaction={decorate}, decoration={markings, mark=at position 0.5 with {\arrow[line width=0.2mm]{angle 90 reversed}}}] (n4) to[out=0, in=180] (n11);
\draw[postaction={decorate}, decoration={markings, mark=at position 0.5 with {\arrow[line width=0.2mm]{angle 90}}}] (n1) to[out=0, in=180] (n12);
\draw[postaction={decorate}, decoration={markings, mark=at position 0.5 with {\arrow[line width=0.2mm]{angle 90 reversed}}}] (n8) to[out=0, in=180] (n15);
\draw[postaction={decorate}, decoration={markings, mark=at position 0.5 with {\arrow[line width=0.2mm]{angle 90}}}] (n9) to[out=0, in=180] (n16);
\draw[postaction={decorate}, decoration={markings, mark=at position 0.5 with {\arrow[line width=0.2mm]{angle 90 reversed}}}] (n13) to[out=0, in=180] (n17);
\draw[postaction={decorate}, decoration={markings, mark=at position 0.5 with {\arrow[line width=0.2mm]{angle 90}}}] (n14) to[out=0, in=180] (n18);
\end{tikzpicture}\end{gathered}\end{equation}

\end{definition}

\begin{note*}

We can see that lawful optics are exactly the homomorphisms for the
``pair-of-pants'' comonoid made from pairs of oriented wires.
Interestingly, if we view this comonoid as a procomonad on \(C\), then
lawful optics are in bijection with its coalgebras on the carrier
\(R_x\). This is a significant generalization of the result by O'Connor
\cite{oconnorLensesAreExactly2010} that lawful lenses are the
coalgebras for the store comonad: here the ``pair-of-pants'' procomonad
precisely generalizes the store comonad.

\end{note*}

\begin{theorem}

\label{lawful_optic} This notion of lawfulness is equivalent to the one
defined by Riley in \cite[Section 3]{rileyCategoriesOptics2018}.

\end{theorem}

\begin{proof}

See appendix \ref{proof_lawful_optic}.

\end{proof}

Thus this diagrammatic definition captures properly the useful and very
general notion of lawfulness for optics. Using this theorem, many
properties of lawfulness can be derived purely diagrammatically. As an
example, let us reprove \cite[Proposition
3.0.4]{rileyCategoriesOptics2018}:

\begin{proposition}[{\cite[Proposition 3.0.4]{rileyCategoriesOptics2018}}]

If \(\alpha\) and \(\beta\) are mutual inverses, then the optic
\(\opval{\alpha}{\beta}{m}\) is lawful.

\end{proposition}

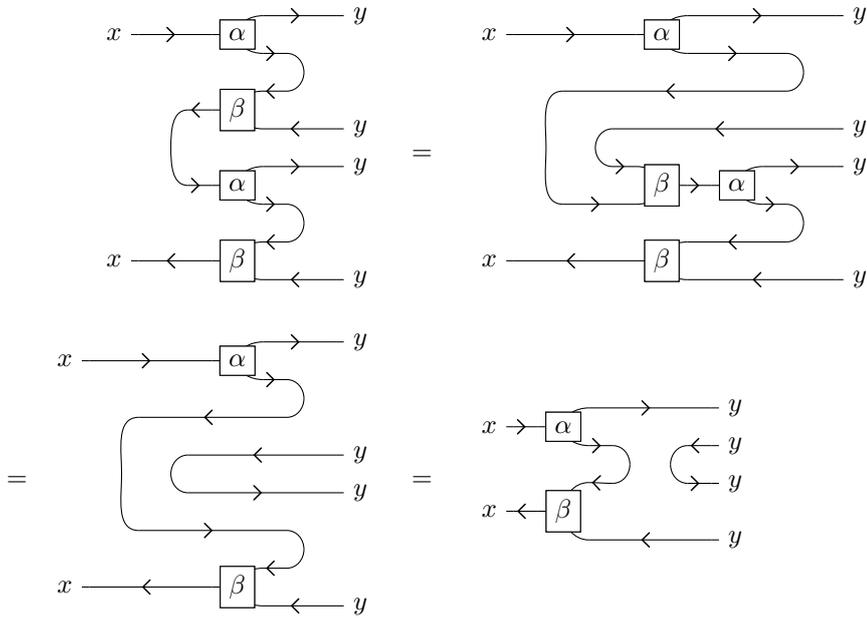
\begin{proof}

\[\begin{array}{rcl}\begin{gathered}\begin{tikzpicture}\matrix[column sep=3mm]{\node at (-0.1,0.25) [anchor=east] {\ensuremath{x}};
\draw[] (-0.1,0.25) to[out=0, in=180] (0.0,0.25);
\path (0.0,0.25) node[coordinate] (n0) {};
\node at (-0.1,1.25) [anchor=east] {\ensuremath{x}};
\draw[] (-0.1,1.25) to[out=0, in=180] (0.0,1.25);
\path (0.0,1.25) node[coordinate] (n1) {};
&
\node at (0.0,0.25) [rectangle, draw, fill=white] (n2) {\ensuremath{\beta{}}};
\path (n2.west) ++(-0.1,0.0) node[coordinate] (n3) {};
\draw[] (n3) to[out=0, in=180] (0.0,0.25);
\path (n2.east) ++(0.1,-0.25) node[coordinate] (n4) {};
\draw[] (0.0,0.25) to[out=down, in=180] (n4);
\path (n2.east) ++(0.1,0.25) node[coordinate] (n5) {};
\draw[] (0.0,0.25) to[out=up, in=180] (n5);
\node at (0.0,0.25) [rectangle, draw, fill=white]  {\ensuremath{\beta{}}};
\node at (0.0,1.25) [rectangle, draw, fill=white] (n6) {\ensuremath{\alpha{}}};
\path (n6.west) ++(-0.1,0.0) node[coordinate] (n7) {};
\draw[] (n7) to[out=0, in=180] (0.0,1.25);
\path (n6.east) ++(0.1,-0.25) node[coordinate] (n8) {};
\draw[] (0.0,1.25) to[out=down, in=180] (n8);
\path (n6.east) ++(0.1,0.25) node[coordinate] (n9) {};
\draw[] (0.0,1.25) to[out=up, in=180] (n9);
\node at (0.0,1.25) [rectangle, draw, fill=white]  {\ensuremath{\alpha{}}};
&
\node at (0.0,0.75) [] (n10) {\ensuremath{}};
\path (n10.west) ++(-0.1,-0.25) node[coordinate] (n11) {};
\draw[] (n11) to[out=0, in=down] (0.0,0.75);
\path (n10.west) ++(-0.1,0.25) node[coordinate] (n12) {};
\draw[] (n12) to[out=0, in=up] (0.0,0.75);
\node at (0.0,0.75) []  {\ensuremath{}};
&
\node at (0.0,0.75) [] (n13) {\ensuremath{}};
\path (n13.west) ++(-0.1,-0.75) node[coordinate] (n14) {};
\draw[] (n14) to[out=0, in=down] (0.0,0.75);
\path (n13.west) ++(-0.1,0.75) node[coordinate] (n15) {};
\draw[] (n15) to[out=0, in=up] (0.0,0.75);
\node at (0.0,0.75) []  {\ensuremath{}};
&
\\};\draw[postaction={decorate}, decoration={markings, mark=at position 0.5 with {\arrow[line width=0.2mm]{angle 90 reversed}}}] (n0) to[out=0, in=180] (n3);
\draw[postaction={decorate}, decoration={markings, mark=at position 0.5 with {\arrow[line width=0.2mm]{angle 90}}}] (n1) to[out=0, in=180] (n7);
\draw[postaction={decorate}, decoration={markings, mark=at position 0.5 with {\arrow[line width=0.2mm]{angle 90 reversed}}}] (n5) to[out=0, in=180] (n11);
\draw[postaction={decorate}, decoration={markings, mark=at position 0.5 with {\arrow[line width=0.2mm]{angle 90}}}] (n8) to[out=0, in=180] (n12);
\draw[postaction={decorate}, decoration={markings, mark=at position 0.5 with {\arrow[line width=0.2mm]{angle 90 reversed}}}] (n4) to[out=0, in=180] (n14);
\draw[postaction={decorate}, decoration={markings, mark=at position 0.5 with {\arrow[line width=0.2mm]{angle 90}}}] (n9) to[out=0, in=180] (n15);
\end{tikzpicture}\end{gathered}&=&\begin{gathered}\begin{tikzpicture}\matrix[column sep=3mm]{\node at (-0.1,0.0) [anchor=east] {\ensuremath{x}};
\draw[] (-0.1,0.0) to[out=0, in=180] (0.0,0.0);
\path (0.0,0.0) node[coordinate] (n0) {};
\node at (-0.1,0.75) [anchor=east] {\ensuremath{x}};
\draw[] (-0.1,0.75) to[out=0, in=180] (0.0,0.75);
\path (0.0,0.75) node[coordinate] (n1) {};
&
\node at (0.0,0.75) [rectangle, draw, fill=white] (n2) {\ensuremath{\alpha{}}};
\path (n2.west) ++(-0.1,0.0) node[coordinate] (n3) {};
\draw[] (n3) to[out=0, in=180] (0.0,0.75);
\path (n2.east) ++(0.1,-0.25) node[coordinate] (n4) {};
\draw[] (0.0,0.75) to[out=down, in=180] (n4);
\path (n2.east) ++(0.1,0.25) node[coordinate] (n5) {};
\draw[] (0.0,0.75) to[out=up, in=180] (n5);
\node at (0.0,0.75) [rectangle, draw, fill=white]  {\ensuremath{\alpha{}}};
&
\node at (0.0,0.75) [rectangle, draw, fill=white] (n6) {\ensuremath{\beta{}}};
\path (n6.west) ++(-0.1,-0.25) node[coordinate] (n7) {};
\draw[] (n7) to[out=0, in=down] (0.0,0.75);
\path (n6.west) ++(-0.1,0.25) node[coordinate] (n8) {};
\draw[] (n8) to[out=0, in=up] (0.0,0.75);
\path (n6.east) ++(0.1,0.0) node[coordinate] (n9) {};
\draw[] (0.0,0.75) to[out=0, in=180] (n9);
\node at (0.0,0.75) [rectangle, draw, fill=white]  {\ensuremath{\beta{}}};
&
\node at (0.0,0.38) [] (n10) {\ensuremath{}};
\path (n10.west) ++(-0.1,-0.38) node[coordinate] (n11) {};
\draw[] (n11) to[out=0, in=down] (0.0,0.38);
\path (n10.west) ++(-0.1,0.38) node[coordinate] (n12) {};
\draw[] (n12) to[out=0, in=up] (0.0,0.38);
\node at (0.0,0.38) []  {\ensuremath{}};
&
\\};\draw[postaction={decorate}, decoration={markings, mark=at position 0.5 with {\arrow[line width=0.2mm]{angle 90}}}] (n1) to[out=0, in=180] (n3);
\draw[postaction={decorate}, decoration={markings, mark=at position 0.5 with {\arrow[line width=0.2mm]{angle 90}}}] (n4) to[out=0, in=180] (n7);
\draw[postaction={decorate}, decoration={markings, mark=at position 0.5 with {\arrow[line width=0.2mm]{angle 90}}}] (n5) to[out=0, in=180] (n8);
\draw[postaction={decorate}, decoration={markings, mark=at position 0.5 with {\arrow[line width=0.2mm]{angle 90 reversed}}}] (n0) to[out=0, in=180] (n11);
\draw[postaction={decorate}, decoration={markings, mark=at position 0.5 with {\arrow[line width=0.2mm]{angle 90}}}] (n9) to[out=0, in=180] (n12);
\end{tikzpicture}\end{gathered}\\&=&\begin{gathered}\begin{tikzpicture}\matrix[column sep=3mm]{\node at (-0.1,0.0) [anchor=east] {\ensuremath{x}};
\draw[] (-0.1,0.0) to[out=0, in=180] (0.0,0.0);
\path (0.0,0.0) node[coordinate] (n0) {};
\node at (-0.1,0.5) [anchor=east] {\ensuremath{x}};
\draw[] (-0.1,0.5) to[out=0, in=180] (0.0,0.5);
\path (0.0,0.5) node[coordinate] (n1) {};
&
\node at (0.0,0.25) [] (n2) {\ensuremath{}};
\path (n2.west) ++(-0.1,-0.25) node[coordinate] (n3) {};
\draw[] (n3) to[out=0, in=down] (0.0,0.25);
\path (n2.west) ++(-0.1,0.25) node[coordinate] (n4) {};
\draw[] (n4) to[out=0, in=up] (0.0,0.25);
\node at (0.0,0.25) []  {\ensuremath{}};
&
\\};\draw[postaction={decorate}, decoration={markings, mark=at position 0.5 with {\arrow[line width=0.2mm]{angle 90 reversed}}}] (n0) to[out=0, in=180] (n3);
\draw[postaction={decorate}, decoration={markings, mark=at position 0.5 with {\arrow[line width=0.2mm]{angle 90}}}] (n1) to[out=0, in=180] (n4);
\end{tikzpicture}\end{gathered}\end{array}\]

\[\begin{array}{rcl}\begin{gathered}\begin{tikzpicture}\matrix[column sep=3mm]{\node at (-0.1,0.25) [anchor=east] {\ensuremath{x}};
\draw[] (-0.1,0.25) to[out=0, in=180] (0.0,0.25);
\path (0.0,0.25) node[coordinate] (n0) {};
\node at (-0.1,3.25) [anchor=east] {\ensuremath{x}};
\draw[] (-0.1,3.25) to[out=0, in=180] (0.0,3.25);
\path (0.0,3.25) node[coordinate] (n1) {};
&
\node at (0.0,1.75) [] (n2) {\ensuremath{}};
\path (n2.east) ++(0.1,-0.5) node[coordinate] (n3) {};
\draw[] (0.0,1.75) to[out=down, in=180] (n3);
\path (n2.east) ++(0.1,0.5) node[coordinate] (n4) {};
\draw[] (0.0,1.75) to[out=up, in=180] (n4);
\node at (0.0,1.75) []  {\ensuremath{}};
&
\node at (0.0,0.25) [rectangle, draw, fill=white] (n5) {\ensuremath{\beta{}}};
\path (n5.west) ++(-0.1,0.0) node[coordinate] (n6) {};
\draw[] (n6) to[out=0, in=180] (0.0,0.25);
\path (n5.east) ++(0.1,-0.25) node[coordinate] (n7) {};
\draw[] (0.0,0.25) to[out=down, in=180] (n7);
\path (n5.east) ++(0.1,0.25) node[coordinate] (n8) {};
\draw[] (0.0,0.25) to[out=up, in=180] (n8);
\node at (0.0,0.25) [rectangle, draw, fill=white]  {\ensuremath{\beta{}}};
\node at (0.0,1.25) [rectangle, draw, fill=white] (n9) {\ensuremath{\alpha{}}};
\path (n9.west) ++(-0.1,0.0) node[coordinate] (n10) {};
\draw[] (n10) to[out=0, in=180] (0.0,1.25);
\path (n9.east) ++(0.1,-0.25) node[coordinate] (n11) {};
\draw[] (0.0,1.25) to[out=down, in=180] (n11);
\path (n9.east) ++(0.1,0.25) node[coordinate] (n12) {};
\draw[] (0.0,1.25) to[out=up, in=180] (n12);
\node at (0.0,1.25) [rectangle, draw, fill=white]  {\ensuremath{\alpha{}}};
\node at (0.0,2.25) [rectangle, draw, fill=white] (n13) {\ensuremath{\beta{}}};
\path (n13.west) ++(-0.1,0.0) node[coordinate] (n14) {};
\draw[] (n14) to[out=0, in=180] (0.0,2.25);
\path (n13.east) ++(0.1,-0.25) node[coordinate] (n15) {};
\draw[] (0.0,2.25) to[out=down, in=180] (n15);
\path (n13.east) ++(0.1,0.25) node[coordinate] (n16) {};
\draw[] (0.0,2.25) to[out=up, in=180] (n16);
\node at (0.0,2.25) [rectangle, draw, fill=white]  {\ensuremath{\beta{}}};
\node at (0.0,3.25) [rectangle, draw, fill=white] (n17) {\ensuremath{\alpha{}}};
\path (n17.west) ++(-0.1,0.0) node[coordinate] (n18) {};
\draw[] (n18) to[out=0, in=180] (0.0,3.25);
\path (n17.east) ++(0.1,-0.25) node[coordinate] (n19) {};
\draw[] (0.0,3.25) to[out=down, in=180] (n19);
\path (n17.east) ++(0.1,0.25) node[coordinate] (n20) {};
\draw[] (0.0,3.25) to[out=up, in=180] (n20);
\node at (0.0,3.25) [rectangle, draw, fill=white]  {\ensuremath{\alpha{}}};
&
\node at (0.0,0.75) [] (n21) {\ensuremath{}};
\path (n21.west) ++(-0.1,-0.25) node[coordinate] (n22) {};
\draw[] (n22) to[out=0, in=down] (0.0,0.75);
\path (n21.west) ++(-0.1,0.25) node[coordinate] (n23) {};
\draw[] (n23) to[out=0, in=up] (0.0,0.75);
\node at (0.0,0.75) []  {\ensuremath{}};
\node at (0.0,2.75) [] (n24) {\ensuremath{}};
\path (n24.west) ++(-0.1,-0.25) node[coordinate] (n25) {};
\draw[] (n25) to[out=0, in=down] (0.0,2.75);
\path (n24.west) ++(-0.1,0.25) node[coordinate] (n26) {};
\draw[] (n26) to[out=0, in=up] (0.0,2.75);
\node at (0.0,2.75) []  {\ensuremath{}};
&
\node at (0.1,0.0) [anchor=west] {\ensuremath{y}};
\draw[] (0.0,0.0) to[out=0, in=180] (0.1,0.0);
\path (0.0,0.0) node[coordinate] (n27) {};
\node at (0.1,1.5) [anchor=west] {\ensuremath{y}};
\draw[] (0.0,1.5) to[out=0, in=180] (0.1,1.5);
\path (0.0,1.5) node[coordinate] (n28) {};
\node at (0.1,2.0) [anchor=west] {\ensuremath{y}};
\draw[] (0.0,2.0) to[out=0, in=180] (0.1,2.0);
\path (0.0,2.0) node[coordinate] (n29) {};
\node at (0.1,3.5) [anchor=west] {\ensuremath{y}};
\draw[] (0.0,3.5) to[out=0, in=180] (0.1,3.5);
\path (0.0,3.5) node[coordinate] (n30) {};
\\};\draw[postaction={decorate}, decoration={markings, mark=at position 0.5 with {\arrow[line width=0.2mm]{angle 90 reversed}}}] (n0) to[out=0, in=180] (n6);
\draw[postaction={decorate}, decoration={markings, mark=at position 0.5 with {\arrow[line width=0.2mm]{angle 90}}}] (n3) to[out=0, in=180] (n10);
\draw[postaction={decorate}, decoration={markings, mark=at position 0.5 with {\arrow[line width=0.2mm]{angle 90 reversed}}}] (n4) to[out=0, in=180] (n14);
\draw[postaction={decorate}, decoration={markings, mark=at position 0.5 with {\arrow[line width=0.2mm]{angle 90}}}] (n1) to[out=0, in=180] (n18);
\draw[postaction={decorate}, decoration={markings, mark=at position 0.5 with {\arrow[line width=0.2mm]{angle 90 reversed}}}] (n8) to[out=0, in=180] (n22);
\draw[postaction={decorate}, decoration={markings, mark=at position 0.5 with {\arrow[line width=0.2mm]{angle 90}}}] (n11) to[out=0, in=180] (n23);
\draw[postaction={decorate}, decoration={markings, mark=at position 0.5 with {\arrow[line width=0.2mm]{angle 90 reversed}}}] (n16) to[out=0, in=180] (n25);
\draw[postaction={decorate}, decoration={markings, mark=at position 0.5 with {\arrow[line width=0.2mm]{angle 90}}}] (n19) to[out=0, in=180] (n26);
\draw[postaction={decorate}, decoration={markings, mark=at position 0.5 with {\arrow[line width=0.2mm]{angle 90 reversed}}}] (n7) to[out=0, in=180] (n27);
\draw[postaction={decorate}, decoration={markings, mark=at position 0.5 with {\arrow[line width=0.2mm]{angle 90}}}] (n12) to[out=0, in=180] (n28);
\draw[postaction={decorate}, decoration={markings, mark=at position 0.5 with {\arrow[line width=0.2mm]{angle 90 reversed}}}] (n15) to[out=0, in=180] (n29);
\draw[postaction={decorate}, decoration={markings, mark=at position 0.5 with {\arrow[line width=0.2mm]{angle 90}}}] (n20) to[out=0, in=180] (n30);
\end{tikzpicture}\end{gathered}&=&\begin{gathered}\begin{tikzpicture}\matrix[column sep=3mm]{\node at (-0.1,0.25) [anchor=east] {\ensuremath{x}};
\draw[] (-0.1,0.25) to[out=0, in=180] (0.0,0.25);
\path (0.0,0.25) node[coordinate] (n0) {};
\node at (-0.1,3.25) [anchor=east] {\ensuremath{x}};
\draw[] (-0.1,3.25) to[out=0, in=180] (0.0,3.25);
\path (0.0,3.25) node[coordinate] (n1) {};
&
\node at (0.0,1.75) [] (n2) {\ensuremath{}};
\path (n2.east) ++(0.1,-0.75) node[coordinate] (n3) {};
\draw[] (0.0,1.75) to[out=down, in=180] (n3);
\path (n2.east) ++(0.1,0.75) node[coordinate] (n4) {};
\draw[] (0.0,1.75) to[out=up, in=180] (n4);
\node at (0.0,1.75) []  {\ensuremath{}};
&
\node at (0.0,1.75) [] (n5) {\ensuremath{}};
\path (n5.east) ++(0.1,-0.25) node[coordinate] (n6) {};
\draw[] (0.0,1.75) to[out=down, in=180] (n6);
\path (n5.east) ++(0.1,0.25) node[coordinate] (n7) {};
\draw[] (0.0,1.75) to[out=up, in=180] (n7);
\node at (0.0,1.75) []  {\ensuremath{}};
&
\node at (0.0,0.25) [rectangle, draw, fill=white] (n8) {\ensuremath{\beta{}}};
\path (n8.west) ++(-0.1,0.0) node[coordinate] (n9) {};
\draw[] (n9) to[out=0, in=180] (0.0,0.25);
\path (n8.east) ++(0.1,-0.25) node[coordinate] (n10) {};
\draw[] (0.0,0.25) to[out=down, in=180] (n10);
\path (n8.east) ++(0.1,0.25) node[coordinate] (n11) {};
\draw[] (0.0,0.25) to[out=up, in=180] (n11);
\node at (0.0,0.25) [rectangle, draw, fill=white]  {\ensuremath{\beta{}}};
\node at (0.0,1.25) [rectangle, draw, fill=white] (n12) {\ensuremath{\beta{}}};
\path (n12.west) ++(-0.1,-0.25) node[coordinate] (n13) {};
\draw[] (n13) to[out=0, in=down] (0.0,1.25);
\path (n12.west) ++(-0.1,0.25) node[coordinate] (n14) {};
\draw[] (n14) to[out=0, in=up] (0.0,1.25);
\path (n12.east) ++(0.1,0.0) node[coordinate] (n15) {};
\draw[] (0.0,1.25) to[out=0, in=180] (n15);
\node at (0.0,1.25) [rectangle, draw, fill=white]  {\ensuremath{\beta{}}};
\node at (0.0,3.25) [rectangle, draw, fill=white] (n16) {\ensuremath{\alpha{}}};
\path (n16.west) ++(-0.1,0.0) node[coordinate] (n17) {};
\draw[] (n17) to[out=0, in=180] (0.0,3.25);
\path (n16.east) ++(0.1,-0.25) node[coordinate] (n18) {};
\draw[] (0.0,3.25) to[out=down, in=180] (n18);
\path (n16.east) ++(0.1,0.25) node[coordinate] (n19) {};
\draw[] (0.0,3.25) to[out=up, in=180] (n19);
\node at (0.0,3.25) [rectangle, draw, fill=white]  {\ensuremath{\alpha{}}};
&
\node at (0.0,1.25) [rectangle, draw, fill=white] (n20) {\ensuremath{\alpha{}}};
\path (n20.west) ++(-0.1,0.0) node[coordinate] (n21) {};
\draw[] (n21) to[out=0, in=180] (0.0,1.25);
\path (n20.east) ++(0.1,-0.25) node[coordinate] (n22) {};
\draw[] (0.0,1.25) to[out=down, in=180] (n22);
\path (n20.east) ++(0.1,0.25) node[coordinate] (n23) {};
\draw[] (0.0,1.25) to[out=up, in=180] (n23);
\node at (0.0,1.25) [rectangle, draw, fill=white]  {\ensuremath{\alpha{}}};
&
\node at (0.0,0.75) [] (n24) {\ensuremath{}};
\path (n24.west) ++(-0.1,-0.25) node[coordinate] (n25) {};
\draw[] (n25) to[out=0, in=down] (0.0,0.75);
\path (n24.west) ++(-0.1,0.25) node[coordinate] (n26) {};
\draw[] (n26) to[out=0, in=up] (0.0,0.75);
\node at (0.0,0.75) []  {\ensuremath{}};
\node at (0.0,2.75) [] (n27) {\ensuremath{}};
\path (n27.west) ++(-0.1,-0.25) node[coordinate] (n28) {};
\draw[] (n28) to[out=0, in=down] (0.0,2.75);
\path (n27.west) ++(-0.1,0.25) node[coordinate] (n29) {};
\draw[] (n29) to[out=0, in=up] (0.0,2.75);
\node at (0.0,2.75) []  {\ensuremath{}};
&
\node at (0.1,0.0) [anchor=west] {\ensuremath{y}};
\draw[] (0.0,0.0) to[out=0, in=180] (0.1,0.0);
\path (0.0,0.0) node[coordinate] (n30) {};
\node at (0.1,1.5) [anchor=west] {\ensuremath{y}};
\draw[] (0.0,1.5) to[out=0, in=180] (0.1,1.5);
\path (0.0,1.5) node[coordinate] (n31) {};
\node at (0.1,2.0) [anchor=west] {\ensuremath{y}};
\draw[] (0.0,2.0) to[out=0, in=180] (0.1,2.0);
\path (0.0,2.0) node[coordinate] (n32) {};
\node at (0.1,3.5) [anchor=west] {\ensuremath{y}};
\draw[] (0.0,3.5) to[out=0, in=180] (0.1,3.5);
\path (0.0,3.5) node[coordinate] (n33) {};
\\};\draw[postaction={decorate}, decoration={markings, mark=at position 0.5 with {\arrow[line width=0.2mm]{angle 90 reversed}}}] (n0) to[out=0, in=180] (n9);
\draw[postaction={decorate}, decoration={markings, mark=at position 0.5 with {\arrow[line width=0.2mm]{angle 90}}}] (n3) to[out=0, in=180] (n13);
\draw[postaction={decorate}, decoration={markings, mark=at position 0.5 with {\arrow[line width=0.2mm]{angle 90}}}] (n6) to[out=0, in=180] (n14);
\draw[postaction={decorate}, decoration={markings, mark=at position 0.5 with {\arrow[line width=0.2mm]{angle 90}}}] (n1) to[out=0, in=180] (n17);
\draw[postaction={decorate}, decoration={markings, mark=at position 0.5 with {\arrow[line width=0.2mm]{angle 90}}}] (n15) to[out=0, in=180] (n21);
\draw[postaction={decorate}, decoration={markings, mark=at position 0.5 with {\arrow[line width=0.2mm]{angle 90 reversed}}}] (n11) to[out=0, in=180] (n25);
\draw[postaction={decorate}, decoration={markings, mark=at position 0.5 with {\arrow[line width=0.2mm]{angle 90}}}] (n22) to[out=0, in=180] (n26);
\draw[postaction={decorate}, decoration={markings, mark=at position 0.5 with {\arrow[line width=0.2mm]{angle 90 reversed}}}] (n4) to[out=0, in=180] (n28);
\draw[postaction={decorate}, decoration={markings, mark=at position 0.5 with {\arrow[line width=0.2mm]{angle 90}}}] (n18) to[out=0, in=180] (n29);
\draw[postaction={decorate}, decoration={markings, mark=at position 0.5 with {\arrow[line width=0.2mm]{angle 90 reversed}}}] (n10) to[out=0, in=180] (n30);
\draw[postaction={decorate}, decoration={markings, mark=at position 0.5 with {\arrow[line width=0.2mm]{angle 90}}}] (n23) to[out=0, in=180] (n31);
\draw[postaction={decorate}, decoration={markings, mark=at position 0.5 with {\arrow[line width=0.2mm]{angle 90 reversed}}}] (n7) to[out=0, in=180] (n32);
\draw[postaction={decorate}, decoration={markings, mark=at position 0.5 with {\arrow[line width=0.2mm]{angle 90}}}] (n19) to[out=0, in=180] (n33);
\end{tikzpicture}\end{gathered}\\=\begin{gathered}\begin{tikzpicture}\matrix[column sep=3mm]{\node at (-0.1,0.25) [anchor=east] {\ensuremath{x}};
\draw[] (-0.1,0.25) to[out=0, in=180] (0.0,0.25);
\path (0.0,0.25) node[coordinate] (n0) {};
\node at (-0.1,3.25) [anchor=east] {\ensuremath{x}};
\draw[] (-0.1,3.25) to[out=0, in=180] (0.0,3.25);
\path (0.0,3.25) node[coordinate] (n1) {};
&
\node at (0.0,1.75) [] (n2) {\ensuremath{}};
\path (n2.east) ++(0.1,-0.75) node[coordinate] (n3) {};
\draw[] (0.0,1.75) to[out=down, in=180] (n3);
\path (n2.east) ++(0.1,0.75) node[coordinate] (n4) {};
\draw[] (0.0,1.75) to[out=up, in=180] (n4);
\node at (0.0,1.75) []  {\ensuremath{}};
&
\node at (0.0,1.75) [] (n5) {\ensuremath{}};
\path (n5.east) ++(0.1,-0.25) node[coordinate] (n6) {};
\draw[] (0.0,1.75) to[out=down, in=180] (n6);
\path (n5.east) ++(0.1,0.25) node[coordinate] (n7) {};
\draw[] (0.0,1.75) to[out=up, in=180] (n7);
\node at (0.0,1.75) []  {\ensuremath{}};
&
\node at (0.0,0.25) [rectangle, draw, fill=white] (n8) {\ensuremath{\beta{}}};
\path (n8.west) ++(-0.1,0.0) node[coordinate] (n9) {};
\draw[] (n9) to[out=0, in=180] (0.0,0.25);
\path (n8.east) ++(0.1,-0.25) node[coordinate] (n10) {};
\draw[] (0.0,0.25) to[out=down, in=180] (n10);
\path (n8.east) ++(0.1,0.25) node[coordinate] (n11) {};
\draw[] (0.0,0.25) to[out=up, in=180] (n11);
\node at (0.0,0.25) [rectangle, draw, fill=white]  {\ensuremath{\beta{}}};
\node at (0.0,3.25) [rectangle, draw, fill=white] (n12) {\ensuremath{\alpha{}}};
\path (n12.west) ++(-0.1,0.0) node[coordinate] (n13) {};
\draw[] (n13) to[out=0, in=180] (0.0,3.25);
\path (n12.east) ++(0.1,-0.25) node[coordinate] (n14) {};
\draw[] (0.0,3.25) to[out=down, in=180] (n14);
\path (n12.east) ++(0.1,0.25) node[coordinate] (n15) {};
\draw[] (0.0,3.25) to[out=up, in=180] (n15);
\node at (0.0,3.25) [rectangle, draw, fill=white]  {\ensuremath{\alpha{}}};
&
\node at (0.0,0.75) [] (n16) {\ensuremath{}};
\path (n16.west) ++(-0.1,-0.25) node[coordinate] (n17) {};
\draw[] (n17) to[out=0, in=down] (0.0,0.75);
\path (n16.west) ++(-0.1,0.25) node[coordinate] (n18) {};
\draw[] (n18) to[out=0, in=up] (0.0,0.75);
\node at (0.0,0.75) []  {\ensuremath{}};
\node at (0.0,2.75) [] (n19) {\ensuremath{}};
\path (n19.west) ++(-0.1,-0.25) node[coordinate] (n20) {};
\draw[] (n20) to[out=0, in=down] (0.0,2.75);
\path (n19.west) ++(-0.1,0.25) node[coordinate] (n21) {};
\draw[] (n21) to[out=0, in=up] (0.0,2.75);
\node at (0.0,2.75) []  {\ensuremath{}};
&
\node at (0.1,0.0) [anchor=west] {\ensuremath{y}};
\draw[] (0.0,0.0) to[out=0, in=180] (0.1,0.0);
\path (0.0,0.0) node[coordinate] (n22) {};
\node at (0.1,1.5) [anchor=west] {\ensuremath{y}};
\draw[] (0.0,1.5) to[out=0, in=180] (0.1,1.5);
\path (0.0,1.5) node[coordinate] (n23) {};
\node at (0.1,2.0) [anchor=west] {\ensuremath{y}};
\draw[] (0.0,2.0) to[out=0, in=180] (0.1,2.0);
\path (0.0,2.0) node[coordinate] (n24) {};
\node at (0.1,3.5) [anchor=west] {\ensuremath{y}};
\draw[] (0.0,3.5) to[out=0, in=180] (0.1,3.5);
\path (0.0,3.5) node[coordinate] (n25) {};
\\};\draw[postaction={decorate}, decoration={markings, mark=at position 0.5 with {\arrow[line width=0.2mm]{angle 90 reversed}}}] (n0) to[out=0, in=180] (n9);
\draw[postaction={decorate}, decoration={markings, mark=at position 0.5 with {\arrow[line width=0.2mm]{angle 90}}}] (n1) to[out=0, in=180] (n13);
\draw[postaction={decorate}, decoration={markings, mark=at position 0.5 with {\arrow[line width=0.2mm]{angle 90 reversed}}}] (n11) to[out=0, in=180] (n17);
\draw[postaction={decorate}, decoration={markings, mark=at position 0.5 with {\arrow[line width=0.2mm]{angle 90}}}] (n3) to[out=0, in=180] (n18);
\draw[postaction={decorate}, decoration={markings, mark=at position 0.5 with {\arrow[line width=0.2mm]{angle 90 reversed}}}] (n4) to[out=0, in=180] (n20);
\draw[postaction={decorate}, decoration={markings, mark=at position 0.5 with {\arrow[line width=0.2mm]{angle 90}}}] (n14) to[out=0, in=180] (n21);
\draw[postaction={decorate}, decoration={markings, mark=at position 0.5 with {\arrow[line width=0.2mm]{angle 90 reversed}}}] (n10) to[out=0, in=180] (n22);
\draw[postaction={decorate}, decoration={markings, mark=at position 0.5 with {\arrow[line width=0.2mm]{angle 90}}}] (n6) to[out=0, in=180] (n23);
\draw[postaction={decorate}, decoration={markings, mark=at position 0.5 with {\arrow[line width=0.2mm]{angle 90 reversed}}}] (n7) to[out=0, in=180] (n24);
\draw[postaction={decorate}, decoration={markings, mark=at position 0.5 with {\arrow[line width=0.2mm]{angle 90}}}] (n15) to[out=0, in=180] (n25);
\end{tikzpicture}\end{gathered}&=&\begin{gathered}\begin{tikzpicture}\matrix[column sep=3mm]{\node at (-0.1,0.38) [anchor=east] {\ensuremath{x}};
\draw[] (-0.1,0.38) to[out=0, in=180] (0.0,0.38);
\path (0.0,0.38) node[coordinate] (n0) {};
\node at (-0.1,1.5) [anchor=east] {\ensuremath{x}};
\draw[] (-0.1,1.5) to[out=0, in=180] (0.0,1.5);
\path (0.0,1.5) node[coordinate] (n1) {};
&
\node at (0.0,0.38) [rectangle, draw, fill=white] (n2) {\ensuremath{\beta{}}};
\path (n2.west) ++(-0.1,0.0) node[coordinate] (n3) {};
\draw[] (n3) to[out=0, in=180] (0.0,0.38);
\path (n2.east) ++(0.1,-0.38) node[coordinate] (n4) {};
\draw[] (0.0,0.38) to[out=down, in=180] (n4);
\path (n2.east) ++(0.1,0.38) node[coordinate] (n5) {};
\draw[] (0.0,0.38) to[out=up, in=180] (n5);
\node at (0.0,0.38) [rectangle, draw, fill=white]  {\ensuremath{\beta{}}};
\node at (0.0,1.5) [rectangle, draw, fill=white] (n6) {\ensuremath{\alpha{}}};
\path (n6.west) ++(-0.1,0.0) node[coordinate] (n7) {};
\draw[] (n7) to[out=0, in=180] (0.0,1.5);
\path (n6.east) ++(0.1,-0.25) node[coordinate] (n8) {};
\draw[] (0.0,1.5) to[out=down, in=180] (n8);
\path (n6.east) ++(0.1,0.25) node[coordinate] (n9) {};
\draw[] (0.0,1.5) to[out=up, in=180] (n9);
\node at (0.0,1.5) [rectangle, draw, fill=white]  {\ensuremath{\alpha{}}};
&
\node at (0.0,1.0) [] (n10) {\ensuremath{}};
\path (n10.west) ++(-0.1,-0.25) node[coordinate] (n11) {};
\draw[] (n11) to[out=0, in=down] (0.0,1.0);
\path (n10.west) ++(-0.1,0.25) node[coordinate] (n12) {};
\draw[] (n12) to[out=0, in=up] (0.0,1.0);
\node at (0.0,1.0) []  {\ensuremath{}};
&
\node at (0.0,1.0) [] (n13) {\ensuremath{}};
\path (n13.east) ++(0.1,-0.25) node[coordinate] (n14) {};
\draw[] (0.0,1.0) to[out=down, in=180] (n14);
\path (n13.east) ++(0.1,0.25) node[coordinate] (n15) {};
\draw[] (0.0,1.0) to[out=up, in=180] (n15);
\node at (0.0,1.0) []  {\ensuremath{}};
&
\node at (0.1,0.0) [anchor=west] {\ensuremath{y}};
\draw[] (0.0,0.0) to[out=0, in=180] (0.1,0.0);
\path (0.0,0.0) node[coordinate] (n16) {};
\node at (0.1,0.75) [anchor=west] {\ensuremath{y}};
\draw[] (0.0,0.75) to[out=0, in=180] (0.1,0.75);
\path (0.0,0.75) node[coordinate] (n17) {};
\node at (0.1,1.25) [anchor=west] {\ensuremath{y}};
\draw[] (0.0,1.25) to[out=0, in=180] (0.1,1.25);
\path (0.0,1.25) node[coordinate] (n18) {};
\node at (0.1,1.75) [anchor=west] {\ensuremath{y}};
\draw[] (0.0,1.75) to[out=0, in=180] (0.1,1.75);
\path (0.0,1.75) node[coordinate] (n19) {};
\\};\draw[postaction={decorate}, decoration={markings, mark=at position 0.5 with {\arrow[line width=0.2mm]{angle 90 reversed}}}] (n0) to[out=0, in=180] (n3);
\draw[postaction={decorate}, decoration={markings, mark=at position 0.5 with {\arrow[line width=0.2mm]{angle 90}}}] (n1) to[out=0, in=180] (n7);
\draw[postaction={decorate}, decoration={markings, mark=at position 0.5 with {\arrow[line width=0.2mm]{angle 90 reversed}}}] (n5) to[out=0, in=180] (n11);
\draw[postaction={decorate}, decoration={markings, mark=at position 0.5 with {\arrow[line width=0.2mm]{angle 90}}}] (n8) to[out=0, in=180] (n12);
\draw[postaction={decorate}, decoration={markings, mark=at position 0.5 with {\arrow[line width=0.2mm]{angle 90 reversed}}}] (n4) to[out=0, in=180] (n16);
\draw[postaction={decorate}, decoration={markings, mark=at position 0.5 with {\arrow[line width=0.2mm]{angle 90}}}] (n14) to[out=0, in=180] (n17);
\draw[postaction={decorate}, decoration={markings, mark=at position 0.5 with {\arrow[line width=0.2mm]{angle 90 reversed}}}] (n15) to[out=0, in=180] (n18);
\draw[postaction={decorate}, decoration={markings, mark=at position 0.5 with {\arrow[line width=0.2mm]{angle 90}}}] (n9) to[out=0, in=180] (n19);
\end{tikzpicture}\end{gathered}\end{array}\]

\end{proof}

\hypertarget{cartesian-lenses}{%
\subsection{Cartesian Lenses}\label{cartesian-lenses}}

\label{cartesian_lenses}

The canonical special case of optics, that we mentioned in
\autoref{example_lenses}, is cartesian lenses. They arise when we
restrict ourselves to \(C = D = M\) and the monoidal product of \(C\) is
cartesian.

In this setting, we have two important gadgets in \(C\): duplication and
deletion, corresponding respectively to the diagonal map
\(C(x, x \times x)\) and the terminal map \(C(x, I)\). Diagrammatically,
we represent them as follows:

\begin{equation*}\begin{gathered}
\end{gathered}\end{equation}

\end{lemma}

\begin{proof}

This corresponds to the standard fact that
\(f = \langle fst \circ f, snd \circ f \rangle\).

\end{proof}

\begin{theorem}

A lens \(l: \icol{x}{u} \arr \icol{y}{v}\) can be expressed as:

%
\]

which has the required shape. We have:

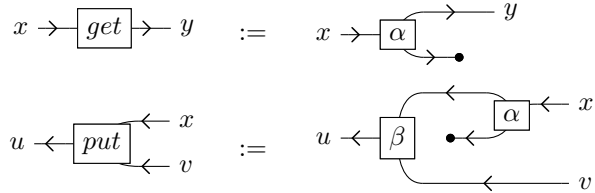
\[%
\]

\end{proof}

\begin{note*}

Observe that it is diagrammatically clear that the definition of \(put\)
and \(get\) in terms of \(\opval{\alpha}{\beta}{m}\) respects the
equivalence relation induced by the coend.

\end{note*}

We recovered purely diagrammatically the usual formulation of lenses in
terms of \(get\) and \(put\), that we had derived in
\autoref{example_lenses}. In this setting, various properties of lenses
can be investigated purely diagrammatically. As an example, let us
revisit \cite[Proposition 3.0.3]{rileyCategoriesOptics2018}, which
captures the fact that the general notion of lawfulness for optics
coincides with the familiar PutGet, GetPut and PutPut laws
\cite{fosterCombinatorsBidirectionalTree2005} (together called
``very-well-behavedness'') in the case of lenses.

\begin{proposition}[{\cite[Proposition 3.0.3]{rileyCategoriesOptics2018}}]

A lens \(l: \icol{x}{x} \arr \icol{y}{y}\) is lawful iff the following
three laws (respectively called PutGet, GetPut and PutPut) hold in C:

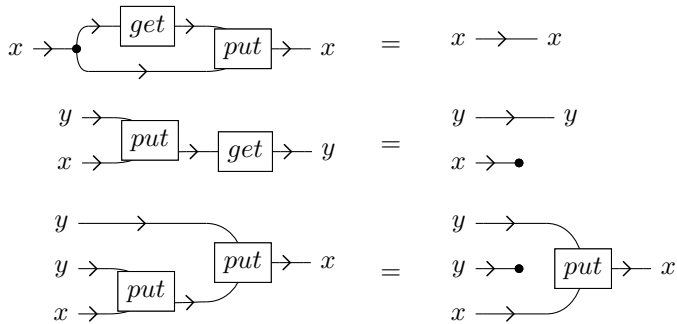
\[
\]

\end{proposition}

\begin{proof}

Diagrammatically, the fact that a lens is lawful reads:

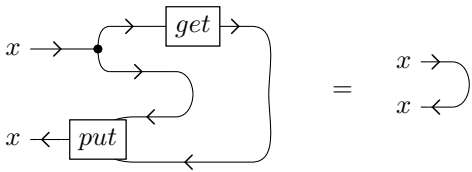
\begin{equation*}\begin{gathered}%
\end{gathered}\end{equation*}

which is exactly the PutGet law, and:

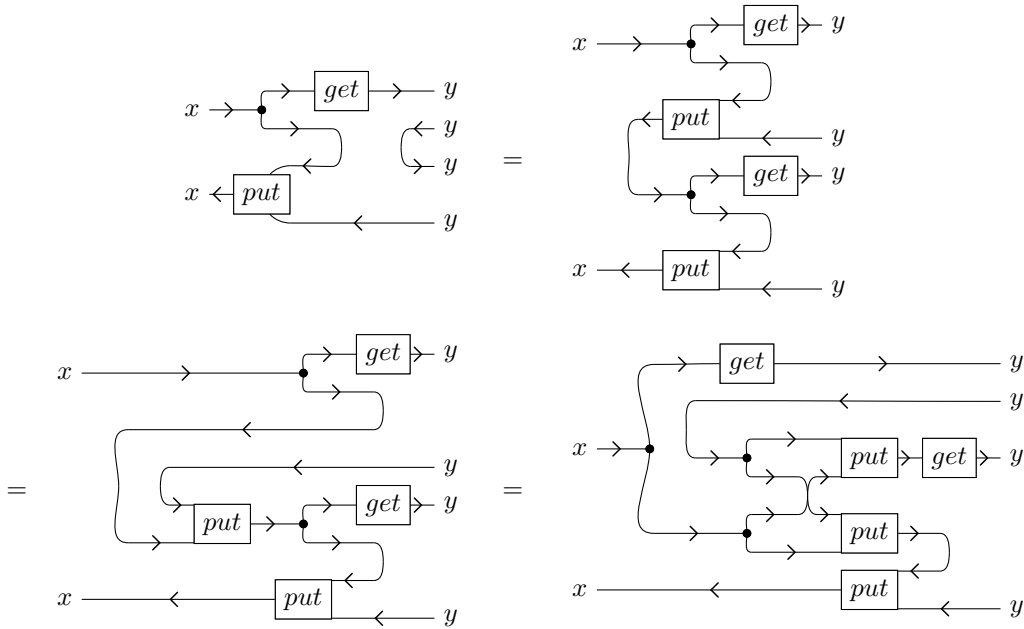
\[%
\]

It is straightforward to see that the PutPut and the GetPut laws
together entail this equality. By applying the deletion map successively
to the outputs, one can also show that this equation entails those two
laws, when \(y\) is inhabited.

\end{proof}

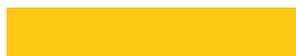
\hypertarget{effectful-lenses}{%
\subsection{Effectful Lenses}\label{effectful-lenses}}

We now turn to a less common example: effectful lenses. They stem from
the desire to allow lenses to perform effects while retrieving or
updating data. Various approaches have been proposed; see Abou-Saleh et
al. \cite{abou-salehReflectionsMonadicLenses2016} for an overview.

Let \(C\) be a cartesian category and \(T\) a monad on \(C\). We would
like an optic that resembles cartesian lenses from the previous section,
but with effectful arrows. This means that we would like our arrows to
live in the Kleisli category \(C_T\). This category however is rarely
monoidal, let alone cartesian: for it to be monoidal, the monad \(T\)
would need to be commutative, which rules out large classes of effects
that we might want to use. Thus we cannot reuse the results from the
previous section. Here we can instead make good use of the generality of
monoidal actions: \(C_T\) may not be monoidal, but when \(T\) is strong
(which is rather common), the product of \(C\) extends to an action of
\(C\) on \(C_T\) \cite[Proposition
4.9.3]{rileyCategoriesOptics2018}. This is enough to define an optic for
monadic lenses:

\[
\MLens_T(\icol{x}{u}, \icol{y}{v}) := \int^{c: C} C_T(x, c \times y) \times C_T(c \times v, u)
\]

Let us now investigate the diagrams for such an optic. Recall the
details of how oriented wires are typed. Here the acting category is
\(C\), which means that in a diagram like the following, the typing
rules enforce that \(x\), \(y\) and \(f\) can live in \(C_T\), but
\(a\), \(b\) and \(g\) can only live in \(C\).

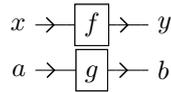
\begin{tikzpicture}\matrix[column sep=3mm]{\node at (-0.1,0.0) [anchor=east] {\ensuremath{a}};
\draw[] (-0.1,0.0) to[out=0, in=180] (0.0,0.0);
\path (0.0,0.0) node[coordinate] (n0) {};
\node at (-0.1,0.6) [anchor=east] {\ensuremath{x}};
\draw[] (-0.1,0.6) to[out=0, in=180] (0.0,0.6);
\path (0.0,0.6) node[coordinate] (n1) {};
&
\node at (0.0,0.0) [rectangle, draw, fill=white] (n2) {\ensuremath{\vphantom{f} g}};
\path (n2.west) ++(-0.1,0.0) node[coordinate] (n3) {};
\draw[] (n3) to[out=0, in=180] (0.0,0.0);
\path (n2.east) ++(0.1,0.0) node[coordinate] (n4) {};
\draw[] (0.0,0.0) to[out=0, in=180] (n4);
\node at (0.0,0.0) [rectangle, draw, fill=white]  {\ensuremath{\vphantom{f} g}};
\node at (0.0,0.6) [rectangle, draw, fill=white] (n5) {\ensuremath{f}};
\path (n5.west) ++(-0.1,0.0) node[coordinate] (n6) {};
\draw[] (n6) to[out=0, in=180] (0.0,0.6);
\path (n5.east) ++(0.1,0.0) node[coordinate] (n7) {};
\draw[] (0.0,0.6) to[out=0, in=180] (n7);
\node at (0.0,0.6) [rectangle, draw, fill=white]  {\ensuremath{f}};
&
\node at (0.1,0.0) [anchor=west] {\ensuremath{b}};
\draw[] (0.0,0.0) to[out=0, in=180] (0.1,0.0);
\path (0.0,0.0) node[coordinate] (n8) {};
\node at (0.1,0.6) [anchor=west] {\ensuremath{y}};
\draw[] (0.0,0.6) to[out=0, in=180] (0.1,0.6);
\path (0.0,0.6) node[coordinate] (n9) {};
\\};\draw[postaction={decorate}, decoration={markings, mark=at position 0.5 with {\arrow[line width=0.2mm]{angle 90}}}] (n0) to[out=0, in=180] (n3);
\draw[postaction={decorate}, decoration={markings, mark=at position 0.5 with {\arrow[line width=0.2mm]{angle 90}}}] (n1) to[out=0, in=180] (n6);
\draw[postaction={decorate}, decoration={markings, mark=at position 0.5 with {\arrow[line width=0.2mm]{angle 90}}}] (n4) to[out=0, in=180] (n8);
\draw[postaction={decorate}, decoration={markings, mark=at position 0.5 with {\arrow[line width=0.2mm]{angle 90}}}] (n7) to[out=0, in=180] (n9);
\end{tikzpicture}

This is why we don't need \(C_T\) to be monoidal: this calculus only
allows an arrow in \(C_T\) to be tensored with arrows in \(C\). This
gives us a string diagram calculus where otherwise none would have been
possible.

The distinction between effectful maps (in \(C_T\)) and pure maps (in
\(C\)) is an important aspect of this calculus. Note that every pure map
\(f\) can be lifted to an effectful map written \(\pure{f}\), via a
canonical functor. This functor also respects the actegory structures,
and therefore allows us to embed the pure lenses from the previous
section as monadic lenses.

This calculus even inherits some of the diagrammatic features of the
previous section: the duplication map and the swap still exist and are
represented as before.

\begin{equation}\begin{gathered}
\end{gathered}\end{equation}

The difference is that the bottom wire can only carry maps living in
\(C\). Whereas before, all maps could be ``slid through'' the
duplication map and swap, now only \(C\)-maps (aka pure maps) can:

\begin{equation}\begin{gathered}%
\end{gathered}\end{equation}

We now restrict ourselves to the monadic lenses proposed by Abou-Saleh
et al. \cite{abou-salehReflectionsMonadicLenses2016}. Those lenses
are modeled closer to ordinary lenses, in particular their definition
does not involve a coend. They are composed of \(get: C(x, y)\) and
\(put: C_T(x \times v, u)\). Diagrammatically they look quite like
cartesian lenses:

%

Note that \(get\) is required to be pure. This is important to ensure
that composing two such lenses stays of that simplified shape. This
non-trivial fact can be seen diagrammatically in what follows: if
\(get\) was not pure, it couldn't be slid across the duplication map.

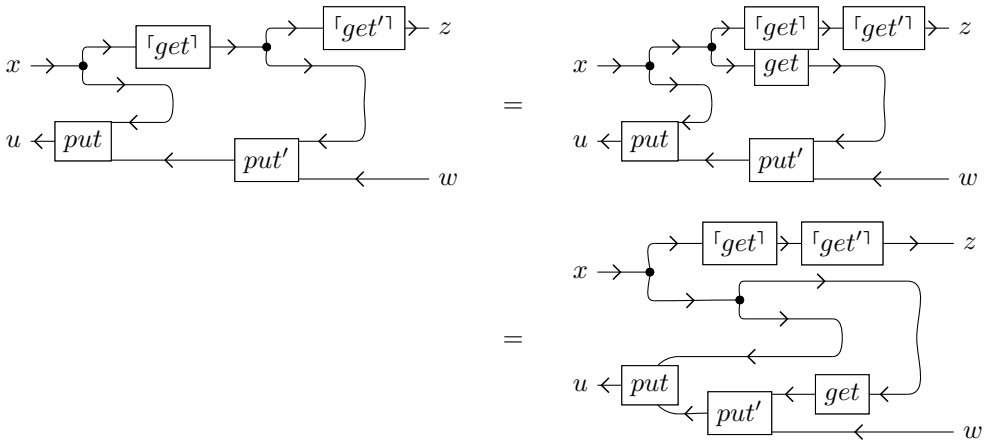
\[%
\]

Finally, this new calculus can express the laws proposed by Abou-Saleh
et al. \cite{abou-salehReflectionsMonadicLenses2016}, making them
much easier to reason about:

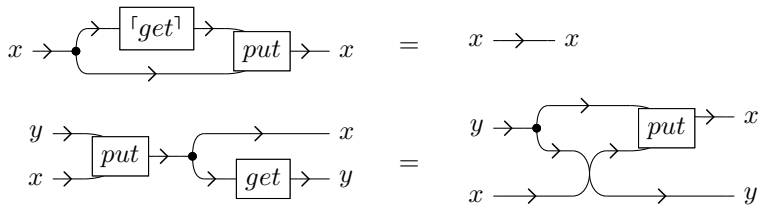
\[%
\]

\hypertarget{conclusion-and-future-work}{%
\section{Conclusion and Future Work}\label{conclusion-and-future-work}}

We have presented a calculus that flowed naturally from the Yoneda
embedding of optics into Tambara modules. We have shown that it was
well-suited for expressing common properties of optics and proving
useful theorems generally, some of which would otherwise be painful to
prove. This work however is only the start: it provides the basis of a
calculus, whose expressive power hasn't yet been explored in the
plethora of topics where optics have found a use. In particular, we
expect new specific diagrammatic properties like those of lenses to
arise for other kinds of optics like prisms or traversals.

Then, the calculus could be linked with related constructions, like the
calculus for teleological categories from
\cite{coherence_for_lenses_and_open_games}, or the Int construction
from \cite{joyalTracedMonoidalCategories1996}.

Properties of \(\Tamb\) as a bicategory also seem worth exploring, in
particular its strong similarity with \(\Prof\), and the link between
the properties of \(M\) and those of \(\Tamb\).

Finally, diagrams in \(\Tamb\) with multiple ingoing and outgoing legs
seem to relate to combs as in
\cite{kissingerCategoricalSemanticsCausal2017} and dialogues in the
style of \cite{hedgesGameSemanticsGame2019}; there is potential for
using \(\Tamb\) to provide a basis for general diagrammatic descriptions
of those objects.

\bibliography{biblio-bibtex}

\appendix

\hypertarget{proofs}{%
\section{Proofs}\label{proofs}}

\hypertarget{r-respects-the-actegory-structure}{%
\subsection{\texorpdfstring{\(R\) Respects the Actegory Structure
(\autoref{R_respects_actegory})}{R Respects the Actegory Structure ()}}\label{r-respects-the-actegory-structure}}

\label{proof_R_respects_actegory}

\begin{proof}[Proof ({\autoref{R_respects_actegory}})]

\[
\begin{aligned}
R_I
&= M(\mathord{-}, \mathord{=} \act_M I)
\\ &= M(\mathord{-}, \mathord{=} \otimes I)
\\ &\cong M(\mathord{-}, \mathord{=})
\\
\\
R_x \otimes R_m
&= \int^{n: M} C(\mathord{-}, n \act_C x) \times M(n, \mathord{=} \act_M m)
\\ &= \int^{n: M} C(\mathord{-}, n \act_C x) \times M(n, \mathord{=} \otimes m)
\\ &\cong C(\mathord{-}, (\mathord{=} \otimes m) \act_C x)
\\ &\cong C(\mathord{-}, \mathord{=} \act_C (m \act_C x))
\\ &= R_{m \act_C x}
\\
\\
L_I
&= M(\mathord{-} \act_M I, \mathord{=})
\\ &= M(\mathord{-} \otimes I, \mathord{=})
\\ &\cong M(\mathord{-}, \mathord{=})
\\
\\
L_m \otimes L_x
&= \int^{n: M} M(\mathord{-} \act_M m, n) \times C(n \act_C x, \mathord{=})
\\ &= \int^{n: M} M(\mathord{-} \otimes m, n) \times C(n \act_C x, \mathord{=})
\\ &\cong C((\mathord{-} \otimes m) \act_C x, \mathord{=})
\\ &\cong C(\mathord{-} \act_C (m \act_C x), \mathord{=})
\\ &= L_{m \act_C x}
\end{aligned}
\]

It is easy to check that the corresponding strengths coincide as well.

\end{proof}

\hypertarget{r-and-l-are-adjoint}{%
\subsection{\texorpdfstring{\(R\) and \(L\) Are Adjoint
(\autoref{prop_bending_wires})}{R and L Are Adjoint ()}}\label{r-and-l-are-adjoint}}

\label{proof_bending_wires}

\begin{proof}[Proof ({\autoref{prop_bending_wires}})]

The counit
\(\varepsilon: R_x \otimes L_x \arr C(\mathord{-}, \mathord{=})\) of the
adjunction in \emph{Prof} is given by composition in \(C\). We need it
to commute with strength:

\[
\begin{tikzcd}
\int^b C(a, b \act x) \otimes C(b \act x, c)
    \arrow{r}{\fatsemi}
    \arrow[swap]{d}{\strength}
& C(a, c)
    \arrow{d}{\strength}
\\%
\int^{b'} C(m \act a, b' \act x) \otimes C(b' \act x, m \act c)
    \arrow{r}{\fatsemi}
& C(m \act a, m \act c)
\end{tikzcd}
\]

We inline the definition of \(\strength\), and move the coends out by
continuity, to get an equivalent square:

\[
\begin{tikzcd}
C(a, b \act x) \otimes C(b \act x, c)
    \arrow{r}{\fatsemi}
    \arrow[swap]{d}{(m \act \mathord{-}) \otimes (m \act \mathord{-})}
& C(a, c)
    \arrow{d}{(m \act \mathord{-})}
\\%
C(m \act a, m \act (b \act x)) \otimes C(m \act (b \act x), m \act c)
    \arrow{r}{\fatsemi}
    \arrow[swap]{d}{C(id, a^{-1}) \otimes C(a, id)}
& C(m \act a, m \act c)
    \arrow{d}{id}
\\%
C(m \act a, (m \otimes b) \act x) \otimes C((m \otimes b) \act x, m \act c)
    \arrow{r}{\fatsemi}
& C(m \act a, m \act c)
\end{tikzcd}
\]

The top square commutes by functoriality of \((m \act \mathord{-})\);
the bottom one by the fact that \(a^{-1} \fatsemi a = id\).

Similarly, the unit also lives in \(\Tamb\). This is enough for the
adjunction to lift from \emph{Prof} to \(\Tamb\).

\end{proof}

\hypertarget{diagram-for-simple-arrows}{%
\subsection{\texorpdfstring{Diagram for Simple Arrows
(\autoref{diagram_embedding})}{Diagram for Simple Arrows ()}}\label{diagram-for-simple-arrows}}

\label{proof_diagram_embedding}

\begin{proof}[Proof ({\autoref{diagram_embedding}})]

The diagram corresponds to the 2-cell \(R_f \otimes L_g\).

It has type \[
\begin{aligned}
R_f \otimes L_g
   &: R_x \otimes L_u \arr R_y \otimes L_v
\\ &= \int_{a b} (\int^m R_x(a, m) \times L_u(m, b)) \arr (\int^m R_y(a, m) \times L_v(m, b))
\end{aligned}
\]

And value \[
\begin{aligned}
(R_f \otimes L_g)(\opval{p}{q}{m})
   &= \opval{R_f(p)}{L_g(q)}{m}
\\ &= \opval{p \fatsemi (m \act f)}{(m \act g) \fatsemi q}{m}
\end{aligned}
\]

To get the preimage through \(Y\), we apply this map to the identity
optic.

\[
\begin{aligned}
(R_f \otimes L_g) (id_{\icol{x}{u}})
&= (R_f \otimes L_g) (\opval{\lambda^{-1}_x}{\lambda_u}{I})
\\ &= \opval{\lambda^{-1}_x \fatsemi (I \act f)}{(I \act g) \fatsemi \lambda_u}{I}
\\ &= \opval{f \fatsemi \lambda^{-1}_y}{\lambda_v \fatsemi g}{I}
\end{aligned}
\]

\end{proof}

\hypertarget{simple-arrows-embed-fully-faithfully}{%
\subsection{\texorpdfstring{Simple Arrows Embed Fully-Faithfully
(\autoref{arrows_ff})}{Simple Arrows Embed Fully-Faithfully ()}}\label{simple-arrows-embed-fully-faithfully}}

\label{proof_arrows_ff}

\begin{proof}[Proof ({\autoref{arrows_ff}})]

We calculate:

\[
\begin{aligned}
& \Optic_{C,M}(\icol{x}{I}, \icol{y}{I})
\\ &= \int^{m} C(x, m \act_C y) \times M(m \act_M I, I)
\\ &= \int^{m} C(x, m \act_C y) \times M(m \otimes I, I)
\\ &\cong \int^{m} C(x, m \act_C y) \times M(m, I)
\\ &\cong C(x, I \act_C y)
\\ &\cong C(x, y)
\end{aligned}
\]

By following the isomorphisms, we get that the reverse direction is the
function
\(f: C(x, y) \mapsto \opval{f \fatsemi \lambda_y^{-1}}{\lambda_I}{I}\),
which as we saw previously corresponds to
\(f \mapsto \iota(f, id_I) = R_f \otimes L_{id_I} = R_f\).

\end{proof}

\hypertarget{representation-theorem}{%
\subsection{\texorpdfstring{Representation Theorem
(\autoref{representation_theorem})}{Representation Theorem ()}}\label{representation-theorem}}

\label{proof_representation_theorem}

\begin{lemma}

The optic corresponding to this diagram is
\(\opval{id_{m \act x}}{id_{m \act u}}{m}\).

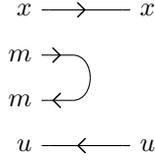
\begin{tikzpicture}\matrix[column sep=3mm]{\node at (-0.1,0.0) [anchor=east] {\ensuremath{u}};
\draw[] (-0.1,0.0) to[out=0, in=180] (0.0,0.0);
\path (0.0,0.0) node[coordinate] (n0) {};
\node at (-0.1,0.6) [anchor=east] {\ensuremath{m}};
\draw[] (-0.1,0.6) to[out=0, in=180] (0.0,0.6);
\path (0.0,0.6) node[coordinate] (n1) {};
\node at (-0.1,1.2) [anchor=east] {\ensuremath{m}};
\draw[] (-0.1,1.2) to[out=0, in=180] (0.0,1.2);
\path (0.0,1.2) node[coordinate] (n2) {};
\node at (-0.1,1.8) [anchor=east] {\ensuremath{x}};
\draw[] (-0.1,1.8) to[out=0, in=180] (0.0,1.8);
\path (0.0,1.8) node[coordinate] (n3) {};
&
\node at (0.0,0.9) [] (n4) {\ensuremath{}};
\path (n4.west) ++(-0.1,-0.3) node[coordinate] (n5) {};
\draw[] (n5) to[out=0, in=down] (0.0,0.9);
\path (n4.west) ++(-0.1,0.3) node[coordinate] (n6) {};
\draw[] (n6) to[out=0, in=up] (0.0,0.9);
\node at (0.0,0.9) []  {\ensuremath{}};
&
\node at (0.1,0.0) [anchor=west] {\ensuremath{u}};
\draw[] (0.0,0.0) to[out=0, in=180] (0.1,0.0);
\path (0.0,0.0) node[coordinate] (n7) {};
\node at (0.1,1.8) [anchor=west] {\ensuremath{x}};
\draw[] (0.0,1.8) to[out=0, in=180] (0.1,1.8);
\path (0.0,1.8) node[coordinate] (n8) {};
\\};\draw[postaction={decorate}, decoration={markings, mark=at position 0.5 with {\arrow[line width=0.2mm]{angle 90 reversed}}}] (n1) to[out=0, in=180] (n5);
\draw[postaction={decorate}, decoration={markings, mark=at position 0.5 with {\arrow[line width=0.2mm]{angle 90}}}] (n2) to[out=0, in=180] (n6);
\draw[postaction={decorate}, decoration={markings, mark=at position 0.5 with {\arrow[line width=0.2mm]{angle 90 reversed}}}] (n0) to[out=0, in=180] (n7);
\draw[postaction={decorate}, decoration={markings, mark=at position 0.5 with {\arrow[line width=0.2mm]{angle 90}}}] (n3) to[out=0, in=180] (n8);
\end{tikzpicture}

\end{lemma}

\begin{proof}

Let us name the map corresponding to this diagram \(\Lambda_{x, m, u}\).

Knowing the action of the cap \(\varepsilon\), we obtain by a tedious
calculation that we will omit here:

\[
\begin{aligned}
\Lambda_{x, m, u}
&: Y \icol{m \act x}{m \act u} \arr Y \icol{x}{u}
\\ &= \opval{\alpha}{\beta}{n} \mapsto
      \opval{\alpha \fatsemi a_{n, m, x}^{-1}}{a_{n, m, u} \fatsemi \beta}{n \otimes m}
\end{aligned}
\]

Thus the corresponding optic is: \[
\begin{aligned}
\Lambda_{x, m, u} (id_{\icol{m \act x}{m \act u}})
&= \Lambda_{x, m, u} (\opval{\lambda^{-1}_{m \act x}}{\lambda_{m \act u}}{I})
\\ &= \opval{\lambda^{-1}_{m \act x} \fatsemi a_{I, m, x}^{-1}}
            {a_{I, m, u} \fatsemi \lambda_{m \act u}}
            {I \otimes m}
\\ &= \opval{\lambda^{-1}_{m} \act x}{\lambda_{m} \act u}{I \otimes m}
\\ &= \opval{(\lambda^{-1}_{m} \fatsemi \lambda_{m}) \act x}{id_{m \act u}}{m}
\\ &= \opval{id_{m \act x}}{id_{m \act u}}{m}
\end{aligned}
\]

\end{proof}

\begin{proof}[Proof ({\autoref{representation_theorem}})]

The right-hand-side diagram is the composition of two optics of which we
know the value: the first is
\(\opval{\alpha \fatsemi \lambda^{-1}_{m \act y}}{\lambda_{m \act v} \fatsemi \beta}{I}\);
the second is \(\opval{id_{m \act y}}{id_{m \act v}}{m}\).

The resulting optic is thus their composition:

\[
\begin{aligned}
& \opval{\alpha \fatsemi \lambda^{-1}_{m \act y}}{\lambda_{m \act v} \fatsemi \beta}{I} \fatsemi \opval{id_{m \act y}}{id_{m \act v}}{m}
\\ &= \opval
    {\alpha \fatsemi \lambda^{-1}_{m \act y} \fatsemi (I \act id_{m \act y}) \fatsemi a^{-1}_{m,I}}
    {a_{m,I} \fatsemi (I \act id_{m \act v}) \fatsemi \lambda_{m \act v} \fatsemi \beta}
    {I \otimes m}
\\ &= \opval
    {\alpha \fatsemi \lambda^{-1}_{m \act y} \fatsemi a^{-1}_{m,I}}
    {a_{m,I} \fatsemi \lambda_{m \act v} \fatsemi \beta}
    {I \otimes m}
\\ &= \opval{\alpha \fatsemi (\lambda^{-1}_{I} \act y)}{(\lambda_{I} \act v) \fatsemi \beta}{I \otimes m}
\\ &= \opval{\alpha \fatsemi (\lambda^{-1}_{I} \act y) \fatsemi (\lambda_{I} \act y)}{\beta}{m}
\\ &= \opval{\alpha}{\beta}{m}
\end{aligned}
\]

\end{proof}

\hypertarget{lawfulness-in-diagrams}{%
\subsection{\texorpdfstring{Lawfulness in Diagrams
(\autoref{lawful_optic})}{Lawfulness in Diagrams ()}}\label{lawfulness-in-diagrams}}

\label{proof_lawful_optic}

\begin{proof}[Proof ({\autoref{lawful_optic}})]

Lawfulness in \cite[Section 3]{rileyCategoriesOptics2018} is based
on three maps named \(outside\), \(once\), and \(twice\). Unpacking the
definitions, those three maps applied to an optic \(l\) correspond
respectively to the three diagrams:

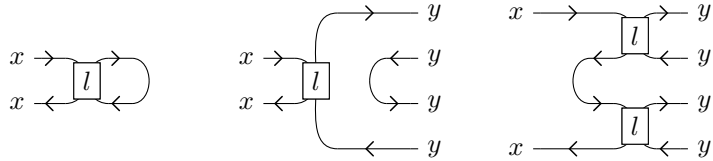
\[\begin{array}{ccc}\begin{gathered}\begin{tikzpicture}\matrix[column sep=3mm]{\node at (-0.1,0.0) [anchor=east] {\ensuremath{x}};
\draw[] (-0.1,0.0) to[out=0, in=180] (0.0,0.0);
\path (0.0,0.0) node[coordinate] (n0) {};
\node at (-0.1,0.6) [anchor=east] {\ensuremath{x}};
\draw[] (-0.1,0.6) to[out=0, in=180] (0.0,0.6);
\path (0.0,0.6) node[coordinate] (n1) {};
&
\node at (0.0,0.3) [rectangle, draw, fill=white] (n2) {\ensuremath{l}};
\path (n2.west) ++(-0.1,-0.3) node[coordinate] (n3) {};
\draw[] (n3) to[out=0, in=down] (0.0,0.3);
\path (n2.west) ++(-0.1,0.3) node[coordinate] (n4) {};
\draw[] (n4) to[out=0, in=up] (0.0,0.3);
\path (n2.east) ++(0.1,-0.3) node[coordinate] (n5) {};
\draw[] (0.0,0.3) to[out=down, in=180] (n5);
\path (n2.east) ++(0.1,0.3) node[coordinate] (n6) {};
\draw[] (0.0,0.3) to[out=up, in=180] (n6);
\node at (0.0,0.3) [rectangle, draw, fill=white]  {\ensuremath{l}};
&
\node at (0.0,0.3) [] (n7) {\ensuremath{}};
\path (n7.west) ++(-0.1,-0.3) node[coordinate] (n8) {};
\draw[] (n8) to[out=0, in=down] (0.0,0.3);
\path (n7.west) ++(-0.1,0.3) node[coordinate] (n9) {};
\draw[] (n9) to[out=0, in=up] (0.0,0.3);
\node at (0.0,0.3) []  {\ensuremath{}};
&
\\};\draw[postaction={decorate}, decoration={markings, mark=at position 0.5 with {\arrow[line width=0.2mm]{angle 90 reversed}}}] (n0) to[out=0, in=180] (n3);
\draw[postaction={decorate}, decoration={markings, mark=at position 0.5 with {\arrow[line width=0.2mm]{angle 90}}}] (n1) to[out=0, in=180] (n4);
\draw[postaction={decorate}, decoration={markings, mark=at position 0.5 with {\arrow[line width=0.2mm]{angle 90 reversed}}}] (n5) to[out=0, in=180] (n8);
\draw[postaction={decorate}, decoration={markings, mark=at position 0.5 with {\arrow[line width=0.2mm]{angle 90}}}] (n6) to[out=0, in=180] (n9);
\end{tikzpicture}\end{gathered}&\begin{gathered}\begin{tikzpicture}\matrix[column sep=3mm]{\node at (-0.1,0.6) [anchor=east] {\ensuremath{x}};
\draw[] (-0.1,0.6) to[out=0, in=180] (0.0,0.6);
\path (0.0,0.6) node[coordinate] (n0) {};
\node at (-0.1,1.2) [anchor=east] {\ensuremath{x}};
\draw[] (-0.1,1.2) to[out=0, in=180] (0.0,1.2);
\path (0.0,1.2) node[coordinate] (n1) {};
&
\node at (0.0,0.9) [rectangle, draw, fill=white] (n2) {\ensuremath{l}};
\path (n2.west) ++(-0.1,-0.3) node[coordinate] (n3) {};
\draw[] (n3) to[out=0, in=down] (0.0,0.9);
\path (n2.west) ++(-0.1,0.3) node[coordinate] (n4) {};
\draw[] (n4) to[out=0, in=up] (0.0,0.9);
\path (n2.east) ++(0.1,-0.9) node[coordinate] (n5) {};
\draw[] (0.0,0.9) to[out=down, in=180] (n5);
\path (n2.east) ++(0.1,0.9) node[coordinate] (n6) {};
\draw[] (0.0,0.9) to[out=up, in=180] (n6);
\node at (0.0,0.9) [rectangle, draw, fill=white]  {\ensuremath{l}};
&
\node at (0.0,0.9) [] (n7) {\ensuremath{}};
\path (n7.east) ++(0.1,-0.3) node[coordinate] (n8) {};
\draw[] (0.0,0.9) to[out=down, in=180] (n8);
\path (n7.east) ++(0.1,0.3) node[coordinate] (n9) {};
\draw[] (0.0,0.9) to[out=up, in=180] (n9);
\node at (0.0,0.9) []  {\ensuremath{}};
&
\node at (0.1,0.0) [anchor=west] {\ensuremath{y}};
\draw[] (0.0,0.0) to[out=0, in=180] (0.1,0.0);
\path (0.0,0.0) node[coordinate] (n10) {};
\node at (0.1,0.6) [anchor=west] {\ensuremath{y}};
\draw[] (0.0,0.6) to[out=0, in=180] (0.1,0.6);
\path (0.0,0.6) node[coordinate] (n11) {};
\node at (0.1,1.2) [anchor=west] {\ensuremath{y}};
\draw[] (0.0,1.2) to[out=0, in=180] (0.1,1.2);
\path (0.0,1.2) node[coordinate] (n12) {};
\node at (0.1,1.8) [anchor=west] {\ensuremath{y}};
\draw[] (0.0,1.8) to[out=0, in=180] (0.1,1.8);
\path (0.0,1.8) node[coordinate] (n13) {};
\\};\draw[postaction={decorate}, decoration={markings, mark=at position 0.5 with {\arrow[line width=0.2mm]{angle 90 reversed}}}] (n0) to[out=0, in=180] (n3);
\draw[postaction={decorate}, decoration={markings, mark=at position 0.5 with {\arrow[line width=0.2mm]{angle 90}}}] (n1) to[out=0, in=180] (n4);
\draw[postaction={decorate}, decoration={markings, mark=at position 0.5 with {\arrow[line width=0.2mm]{angle 90 reversed}}}] (n5) to[out=0, in=180] (n10);
\draw[postaction={decorate}, decoration={markings, mark=at position 0.5 with {\arrow[line width=0.2mm]{angle 90}}}] (n8) to[out=0, in=180] (n11);
\draw[postaction={decorate}, decoration={markings, mark=at position 0.5 with {\arrow[line width=0.2mm]{angle 90 reversed}}}] (n9) to[out=0, in=180] (n12);
\draw[postaction={decorate}, decoration={markings, mark=at position 0.5 with {\arrow[line width=0.2mm]{angle 90}}}] (n6) to[out=0, in=180] (n13);
\end{tikzpicture}\end{gathered}&\begin{gathered}\begin{tikzpicture}\matrix[column sep=3mm]{\node at (-0.1,0.0) [anchor=east] {\ensuremath{x}};
\draw[] (-0.1,0.0) to[out=0, in=180] (0.0,0.0);
\path (0.0,0.0) node[coordinate] (n0) {};
\node at (-0.1,1.8) [anchor=east] {\ensuremath{x}};
\draw[] (-0.1,1.8) to[out=0, in=180] (0.0,1.8);
\path (0.0,1.8) node[coordinate] (n1) {};
&
\node at (0.0,0.9) [] (n2) {\ensuremath{}};
\path (n2.east) ++(0.1,-0.3) node[coordinate] (n3) {};
\draw[] (0.0,0.9) to[out=down, in=180] (n3);
\path (n2.east) ++(0.1,0.3) node[coordinate] (n4) {};
\draw[] (0.0,0.9) to[out=up, in=180] (n4);
\node at (0.0,0.9) []  {\ensuremath{}};
&
\node at (0.0,0.3) [rectangle, draw, fill=white] (n5) {\ensuremath{l}};
\path (n5.west) ++(-0.1,-0.3) node[coordinate] (n6) {};
\draw[] (n6) to[out=0, in=down] (0.0,0.3);
\path (n5.west) ++(-0.1,0.3) node[coordinate] (n7) {};
\draw[] (n7) to[out=0, in=up] (0.0,0.3);
\path (n5.east) ++(0.1,-0.3) node[coordinate] (n8) {};
\draw[] (0.0,0.3) to[out=down, in=180] (n8);
\path (n5.east) ++(0.1,0.3) node[coordinate] (n9) {};
\draw[] (0.0,0.3) to[out=up, in=180] (n9);
\node at (0.0,0.3) [rectangle, draw, fill=white]  {\ensuremath{l}};
\node at (0.0,1.5) [rectangle, draw, fill=white] (n10) {\ensuremath{l}};
\path (n10.west) ++(-0.1,-0.3) node[coordinate] (n11) {};
\draw[] (n11) to[out=0, in=down] (0.0,1.5);
\path (n10.west) ++(-0.1,0.3) node[coordinate] (n12) {};
\draw[] (n12) to[out=0, in=up] (0.0,1.5);
\path (n10.east) ++(0.1,-0.3) node[coordinate] (n13) {};
\draw[] (0.0,1.5) to[out=down, in=180] (n13);
\path (n10.east) ++(0.1,0.3) node[coordinate] (n14) {};
\draw[] (0.0,1.5) to[out=up, in=180] (n14);
\node at (0.0,1.5) [rectangle, draw, fill=white]  {\ensuremath{l}};
&
\node at (0.1,0.0) [anchor=west] {\ensuremath{y}};
\draw[] (0.0,0.0) to[out=0, in=180] (0.1,0.0);
\path (0.0,0.0) node[coordinate] (n15) {};
\node at (0.1,0.6) [anchor=west] {\ensuremath{y}};
\draw[] (0.0,0.6) to[out=0, in=180] (0.1,0.6);
\path (0.0,0.6) node[coordinate] (n16) {};
\node at (0.1,1.2) [anchor=west] {\ensuremath{y}};
\draw[] (0.0,1.2) to[out=0, in=180] (0.1,1.2);
\path (0.0,1.2) node[coordinate] (n17) {};
\node at (0.1,1.8) [anchor=west] {\ensuremath{y}};
\draw[] (0.0,1.8) to[out=0, in=180] (0.1,1.8);
\path (0.0,1.8) node[coordinate] (n18) {};
\\};\draw[postaction={decorate}, decoration={markings, mark=at position 0.5 with {\arrow[line width=0.2mm]{angle 90 reversed}}}] (n0) to[out=0, in=180] (n6);
\draw[postaction={decorate}, decoration={markings, mark=at position 0.5 with {\arrow[line width=0.2mm]{angle 90}}}] (n3) to[out=0, in=180] (n7);
\draw[postaction={decorate}, decoration={markings, mark=at position 0.5 with {\arrow[line width=0.2mm]{angle 90 reversed}}}] (n4) to[out=0, in=180] (n11);
\draw[postaction={decorate}, decoration={markings, mark=at position 0.5 with {\arrow[line width=0.2mm]{angle 90}}}] (n1) to[out=0, in=180] (n12);
\draw[postaction={decorate}, decoration={markings, mark=at position 0.5 with {\arrow[line width=0.2mm]{angle 90 reversed}}}] (n8) to[out=0, in=180] (n15);
\draw[postaction={decorate}, decoration={markings, mark=at position 0.5 with {\arrow[line width=0.2mm]{angle 90}}}] (n9) to[out=0, in=180] (n16);
\draw[postaction={decorate}, decoration={markings, mark=at position 0.5 with {\arrow[line width=0.2mm]{angle 90 reversed}}}] (n13) to[out=0, in=180] (n17);
\draw[postaction={decorate}, decoration={markings, mark=at position 0.5 with {\arrow[line width=0.2mm]{angle 90}}}] (n14) to[out=0, in=180] (n18);
\end{tikzpicture}\end{gathered}\end{array}\]

The interesting insight is that the complicated \(Optic^2_M\) coend from
Riley's paper can be easily constructed diagrammatically by tensoring
oriented wires as above. The theorem then follows directly from Riley's
definition of lawfulness.

\end{proof}

\end{document}